\documentclass[12pt]{amsart}
\usepackage{amsthm,amsmath,amssymb,amscd,amsfonts}

\begin{document}

\title[Generalized canonical rings]
{On twisted Kawamata's semi-positivity  and  finite generation of generalized canonical rings}

\author{Yoshinori  Gongyo}
\author{Shigeharu Takayama}

\address{Graduate School of Mathematical Sciences, 
the University of Tokyo, 3-8-1 Komaba, Meguro-ku, Tokyo 153-8914, Japan.}

\email{gongyo@ms.u-tokyo.ac.jp}
\email{taka@ms.u-tokyo.ac.jp}

\begin{abstract}We give a twisted version of the Kawamata semi-positivity theorem by the $\mathbb{Q}$-line bundle with a vanishing Lelong number at every point. Moreover, we apply the result to the finite generation problem for canonical rings of Birkar's generalized pairs. 

\end{abstract}

\makeatletter 
\@namedef{subjclassname@2020}{\textup{2020} Mathematics Subject Classification}


\date{\today}
\maketitle
\baselineskip=16pt

\theoremstyle{plain}
  \newtheorem{thm}{Theorem}[section]
  \newtheorem{thm'}{Theorem}[subsection]
  \newtheorem{prop}[thm]{Proposition}
  \newtheorem{prop'}[thm']{Proposition}
  \newtheorem{lem}[thm]{Lemma}
  \newtheorem{lem'}[thm']{Lemma}  
  \newtheorem{cor}[thm]{Corollary}
\theoremstyle{definition}  
  \newtheorem{dfn}[thm]{Definition}
  \newtheorem{dfn'}[thm']{Definition}  
  \newtheorem{exmp}[thm]{Example}
  \newtheorem{prob}[thm]{Problem}
  \newtheorem{notation}[thm]{Notation}
  \newtheorem{notation'}[thm']{Notation}
  \newtheorem{quest}[thm]{Question}
  \newtheorem{rem}[thm]{Remark}
  \newtheorem{rem'}[thm']{Remark}  
  \newtheorem*{acknowledgement}{Acknowledgement}
  \newtheorem{no}[thm]{}
  \newtheorem{no'}[thm']{}
  \newtheorem{setup}[thm]{Set up}
  \newtheorem{setup'}[thm']{Set up}  
\numberwithin{equation}{subsection}

\newcommand{\disp}{\displaystyle}

\newcommand{\BC}{{\mathbb{C}}}
\newcommand{\BN}{{\mathbb{N}}}
\newcommand{\BP}{{\mathbb{P}}}
\newcommand{\BQ}{{\mathbb{Q}}}
\newcommand{\BR}{{\mathbb{R}}}
\newcommand{\BZ}{{\mathbb{Z}}}
\newcommand{\BD}{{\mathbb{D}}}

\newcommand{\CA}{{\mathcal{A}}}
\newcommand{\CB}{{\mathcal{B}}}
\newcommand{\CE}{{\mathcal{E}}}
\newcommand{\CF}{{\mathcal{F}}}
\newcommand{\CG}{{\mathcal{G}}}
\newcommand{\CH}{{\mathcal{H}}}
\newcommand{\CI}{{\mathcal{I}}}
\newcommand{\CJ}{{\mathcal{J}}}
\newcommand{\CL}{{\mathcal{L}}}
\newcommand{\CM}{{\mathcal{M}}}
\newcommand{\CO}{{\mathcal{O}}}
\newcommand{\CQ}{{\mathcal{Q}}}
\newcommand{\CS}{{\mathcal{S}}}
\newcommand{\CU}{{\mathcal{U}}}
\newcommand{\cV}{{\mathcal{V}}}
\newcommand{\CY}{{\mathcal{Y}}}

\newcommand{\bfa}{\mathbf {a}}
\newcommand{\bfb}{\mathbf {b}}
\newcommand{\bfc}{\mathbf {c}}
\newcommand{\bfd}{\mathbf {d}}
\newcommand{\bfh}{\mathbf {h}}
\newcommand{\bfk}{\mathbf {k}}
\newcommand{\bfr}{\mathbf {r}}
\newcommand{\bft}{\mathbf {t}}
\newcommand{\bfx}{\mathbf {x}}
\newcommand{\bfy}{\mathbf {y}}
\newcommand{\bfz}{\mathbf {z}}
\newcommand{\bfo}{\mathbf {o}} 

\newcommand{\bfgc}{\mathbf {\gc}}
\newcommand{\bflam}{\mathbf {\lambda}}
\newcommand{\bfth}{\mathbf {\theta}}

\newcommand{\fm}{{\mathfrak {m}}}

\newcommand{\ga}{{\alpha}}
\newcommand{\gb}{{\beta}}
\newcommand{\gc}{{\gamma}}
\newcommand\ep{\varepsilon}
\newcommand{\Ga}{{\Gamma}}
\newcommand{\Del}{{\Delta}}
\newcommand{\Dl}{{\Delta}}
\newcommand{\del}{{\delta}}
\newcommand{\ka}{{\kappa}}
\newcommand{\lam}{{\lambda}}
\newcommand{\Lam}{{\Lambda}}
\newcommand{\vph}{{\varphi}}
\newcommand{\sg}{{\sigma}}
\newcommand{\Th}{{\Theta}}
\newcommand\w{\omega}
\newcommand\Om{\Omega}

\newcommand\nab{\nabla}

\newcommand\wtil{\widetilde}
\newcommand\what{\widehat}
\newcommand\tb{\widetilde b}
\newcommand\tB{\widetilde B}
\newcommand\tD{\widetilde D}
\newcommand\tE{\widetilde E}
\newcommand\tf{\widetilde f}
\newcommand\tF{\widetilde F}
\newcommand\tG{\widetilde G}
\newcommand\tH{\widetilde H}
\newcommand\tk{\widetilde k}
\newcommand\tL{\widetilde L}
\newcommand\tLam{\widetilde \Lam}
\newcommand\tN{\widetilde N}
\newcommand\tP{\widetilde P}
\newcommand\tR{\widetilde R}
\newcommand\tS{\widetilde S}
\newcommand\tu{{\widetilde u}}
\newcommand\tU{\widetilde U}
\newcommand\tv{\widetilde v}
\newcommand\tW{\widetilde W}
\newcommand\tw{\widetilde \w}

\newcommand\tx{\widetilde x}
\newcommand\tX{\widetilde X}

\newcommand\tCM{\widetilde \CM}
\newcommand\teta{\widetilde \eta}
\newcommand\ttau{\widetilde \tau}

\newcommand\sm{\setminus}
\newcommand\ol{\overline}
\newcommand\ot{\otimes}
\newcommand\wed{\wedge}

\newcommand\sumn{\sum\nolimits}

\newcommand\lra{\longrightarrow}
\newcommand\teigi{\Leftrightarrow}
\newcommand\Ra{\Rightarrow}
\newcommand\La{\Leftarrow}

\newcommand\da{\downarrow}

\newcommand\ai{\sqrt{-1}}
\newcommand\rd{{\partial}}
\newcommand\rdb{{\overline{\partial}}}

\newcommand\Amp{\mbox{{\rm Amp}}\, }
\newcommand\codim{\text{{\rm codim}}\, } 
\newcommand\cpt{{\rm cpt}}
\newcommand\diam{\mbox{{\rm diam}}\,}
\newcommand\divisor{\mbox{{\rm div}}\,}
\newcommand\Jac{\mbox{{\rm Jac}}\, }
\newcommand\NAmp{\mbox{{\rm NAmp}}\, }
\newcommand\Proj{\mbox{{\rm Proj}}\, }
\newcommand\pr{\mbox{{\rm pr}}\, }
\newcommand\rank{\mbox{{\rm rank}}\, }
\newcommand\Reg{{\rm Reg}\, }
\newcommand\reg{{\rm reg}}
\newcommand\SBs{\mbox{{\rm SBs}}\, }
\newcommand\Sing{{\rm Sing}\, }
\newcommand\supp{\mbox{{\rm supp}}\, }
\newcommand\Supp{\mbox{{\rm Supp}}\, }
\newcommand\Tr{\mbox{\rm Tr}}
\newcommand\Ty{\mbox{{\rm Ty}}}
\newcommand\vol{\mbox{{\rm vol}}}
\newcommand\Vol{\mbox{{\rm Vol}}\, }

\renewcommand\Re{\mbox{{\rm Re}}\, }

\newcommand{\ru}[1]{\ulcorner{#1}\urcorner}

\section{Introduction}

Throughout this paper, we work over $\mathbb{C}$, the complex number field.

In this paper, we prove the following theorem, which we  call "twisted Kawamata's semi-positivity theorem":

\begin{thm}[Twisted Kawamata's semi-positivity theorem]\label{tKsp}
Let $f: X \to Y$ be a surjective morphism of projective varieties with connected fibers. 
Suppose $X$ is quasi-smooth and $Y$ is smooth.
Let $P=\sum_{j} P_j$ and $Q=\sum_{l} Q_l$ be reduced effective divisors on $X$ and $Y$, respectively, 
such that $f^{-1}(Q)\subseteq P$. Suppose that $f$ is  a smooth morphism over $Y \setminus \mathrm{Supp}\,Q$

Let $M$ be a $\BQ$-divisor on $X$, which, as a $\mathbb{Q}$-line bundle, admits a singular Hermitian metric with semi-positive curvature and whose local weight functions have vanishing Lelong numbers at any points in $X$ (see \ref{defLe=0} for the definition). 
In particular, $M$ is nef.

Let $D=\sum_{j} d_jP_j$ be 
a $\mathbb Q$-divisor {\em{(}}$d_j$'s may 
be negative or zero{\em{)}}, which satisfies 
the following conditions (1)--(4){\em{:}}
\begin{itemize}
\item[(1)] 
$f:(X, P) \to (Y, Q)$ is a well prepared model ($f:(X, P) \to (Y, Q)$ is toroidal and $f$ is equi-dimensional, in fact $f$ is flat),
\item[(2)] $d_j<1$, 
\item[(3)] 
$\dim_{\mathbb C(\eta)}f_*\mathcal O_X(\ulcorner -D\urcorner)\otimes_{\mathcal O_Y}\mathbb C(\eta)=1$, 
where $\eta$ is the generic point of $Y$, and
\item[(4)] 
$K_X+D+M \sim_{\mathbb Q} f^*(K_Y+L)$ 
for some $\mathbb Q$-divisor $L$ on $Y$. 
\end{itemize}

Let {\rm (5)}
\begin{eqnarray} 
f^*Q_l &=& 
\sum_{j}w_{lj}P_j, \ \text{where}\ w_{lj} \in \BZ_{>0}, \nonumber\\
\bar{d}_j&=&\frac{d_j+w_{lj}-1}{w_{lj}} (\in \BQ_{<1}) \ \text{if}\ f(P_j)=Q_l, \nonumber\\
\delta_{l}&=& \mathrm{max}\{\bar{d_j}| f(P_j)=Q_l\}, \nonumber\\
\Delta_{Y}&=&\sum \delta_{l}Q_{l},\ \text{and} \nonumber\\
M_Y&=&L-\Delta_{Y}.\nonumber
\end{eqnarray}
By construction, the pair $(Y,\Del_Y)$ is sub-klt and $K_X+D+M\sim_\BQ f^*(K_Y+\Del_Y+M_Y)$. 
We sometimes call $M_Y$ {\em{(}}resp.~$\Delta_Y${\em{)}} 
the {\em{moduli part}} {\em{(}}resp.~{\em{discriminant 
part}}{\em{)}}.  

Then, the $\BQ$-divisor $M_Y$ admits, as a $\BQ$-line bundle, a singular Hermitian metric with semi-positive curvature and whose local weight functions have vanishing Lelong numbers at any points in $Y$. 
In particular, $M_Y$ is nef. 
\end{thm}

\begin{rem}\label{rem_rel.nc}
In Theorem \ref{tKsp}, the toroidal conditions and the smoothness of $f$ over $Y \setminus \mathrm{Supp}\,Q$ imply that $\mathrm{Supp}\,D^h$ is relative normal crossing over $Y \setminus \mathrm{Supp}\,Q$ since the torus invariant boundaries of smooth toric varieties are simple normal crossings.
\end{rem}

The property (3) is, in fact, to say that $f_{*}\CO_{X} (\ru{-D})$ is a line bundle. When $M=0$ and $X$ is smooth, the above theorem follows from the traditional Kawamata's semi-positivity theorem in \cite{Ksub}. Compared with \cite{Ksub}, we add the extra assumption (1). The condition is required by analysis of the singularities of the metric on $M_Y$. However, applying the toroidalization theorem in \cite{AK}, we can always assume (1) with the birational modification. Then, we need to generalize the theorem for quasi-smooth and toroidal. Therefore, we state the above situation. The proof follows from the idea of the second author's paper \cite{Ta}. However, by the contribution of $M$, the estimation becomes more complicated than   \cite{Ta}. 

Next, using  twisted Kawamata's semi-positivity theorem \ref{tKsp}, we prove the following theorem:

\begin{thm}[{=Theorem \ref{fg-canonical}}]\label{fg-canonical-intro}
Let $(X, B+M)$ be a projective generalized klt pair such that a $\mathbb{Q}$-line bundle  $M'$ on a smooth higher birational model $X' \to X$ has a semi-positive singular Hermitian metric $h_{M'}$ with a vanishing  Lelong number for the potential at every point of $X'$. Assume that $l(K_X+B+M)$ is a Cartier divisor for an $l>0$. 

Then the ring 
$$ \bigoplus_{m \geq 0} H^0(X, ml(K_X+B+M))  
$$
is a finite generated $\mathbb{C}$-algebra. 
\end{thm}

We collect the notion of generalized pairs by Birkar in Section \ref{fg-section}. When $K_X+B+M$ is big, the above theorem directly follows from \cite{bchm} since we perturb $B+M$ with $K_X+B+M$ and produce a usual klt boundary $B'$ such that $B'\sim_{\mathbb{Q}} B+M+\varepsilon(K_X+B+M)$ for a small $\varepsilon>0$. 
Thus, we apply the Fujino--Mori argument \cite{fm}. 

Finally, we note the interesting difference between the usual and generalized pair in the Kawamata semi-positivity, the canonical bundle formula, and the finite generation. 
For generalized pairs, we do not know the functorial behavior of the moduli part of \cite[Theorem 0.2 (1)]{Ajdg}. 
For the case of (generalized) log Calabi--Yau fibrations, it is known by \cite{filipazzi}. 
However, in general, it is unknown. Fortunately, this does not matter for our application. Moreover, one may ask the case of log canonical in Theorem \ref{fg-canonical-intro}. However, the finite generation of lc generalized canonical rings does not hold by \cite[Remark 3.5]{gon}. Our result is also obtained in \cite[Theorem 1.1]{Kim2} when $M$ is trivial
and the horizontal part of $D$ is effective. After putting this paper on the arXiv, we know that  Hacon and P\u{a}un are working with the similar idea with us in \cite{hp}. They also treat various applications for the adjunction theory. 

\section*{Acknowledgement}
The first author was partially supported by grants JSPS KAKENHI $\#$18H01108, 19KK0345, 20H00111, 21H00970, 22H01112, 23H01064. The second author was partially supported by grants JSPS KAENHI $\#$21H00979, $\#$23K20792. 
The first author thanks Professor Paolo Cascini and Chi-Kang Chang for the discussion.

\section{Preliminary}

\subsection{Algebraic preliminary}

We organize some notations and preliminary discussions for our main argument.
Let everything be as in \ref{tKsp}.

\begin{notation'}\label{nota1}
(1)
We write, in general, as $D= D_{+} - D_{-}$ with the positive part $D_{+}=\sum_{d_{i}>0} d_{i}P_{i}$ and the negative part $D_{-}=\sum_{d_{i}<0} (-d_{i})P_{i}$ (without common components).
We have $D=D^{h}_{+} - D^{h}_{-} + D^{v}_{+} - D^{v}_{-}$.

(2)
We let $\Del_{D}=D_{+} + \ru{D_{-}} - D_{-}$. 
It is effective, and the pair $(X,\Del_{D})$ is klt by \ref{tKsp}\,(2). 
We see
$D = \Del_{D} - \ru{D_{-}}$, $\ru{-D} = \ru{D_{-}}$, and hence
$K_{X}+D = K_{X}+ \Del_{D} - \ru{D_{-}}$,
$K_{X/Y}+\Del_{D}+M \sim_{\BQ} \ru{D_{-}} + f^{*}L$.

If $D=\sum d_jP_j$ with $d_j\in \BQ$ in general, $(X,D)$ is sub-klt if and only if all $d_j<1$ ($V$-normal crossing support with coefficients $<1$),
\cite[5.20\,(4)]{komo}.
\end{notation'}

We recall the notion of singularities of pairs. 

\begin{dfn'}[Singularities of pairs] 
Let $X$ be a normal variety and let $\Delta$ be an $\mathbb{R}$-divisor 
on $X$ such that $K_X+\Delta$ is $\mathbb{R}$-Cartier. 
Let $\varphi:Y\rightarrow X$ be a log resolution of $(X,\Delta)$. 
We set $$K_Y=\varphi^*(K_X+\Delta)+\sum a_iE_i,$$ where $E_i$ is a 
prime divisor on $Y$ for every $i$.
The pair $(X,\Delta)$ is called 
\begin{itemize}
\item[(a)] \emph{sub Kawamata log terminal} $($\emph{subklt}, 
for short$)$ if $a_i > -1$ for all $i$, or
\item[(b)]\emph{sub log canonical} $($\emph{sublc}, for short$)$ if $a_i \geq -1$ for all $i$.
\end{itemize}
If $\Delta$ is effective and $(X, \Delta)$ is subklt (resp.~sublc), 
then we  simply call it \emph{klt} (resp.~\emph{lc}). 
\end{dfn'}

Let us recall, in general, a resolution of singularities theorem of morphisms due to Abramovich and Karu \cite{AK}.
(We refer \cite[Theorem 2]{Ka15}. See \cite[\S 1]{Ka15} for terminology.)

\begin{thm'}\label{wp}
Let $f: X\to Y$ be a surjective morphism of smooth projective varieties with connected fibers and $Z$ a closed subset of $X$.
Then there exists a quasi-smooth projective toroidal variety $(X',B')$, a smooth projective toroidal variety $(Y',C')$, a projective morphism $f':X'\to Y'$ with connected fibers and projective birational morphisms $v_{X}:X' \to X, v_{Y}:Y' \to Y$ such that $v_{Y}\circ f' = f\circ v_{X}$ and which satisfy the following conditions.

(1) $f':(X',B')\to (Y',C')$ is a toroidal morphism.

(2) All the fibers of $f'$ have the same dimension ($f'$ is flat in fact).

(3) $v_{X}^{-1}(Z) \subset B'$.

(4) The general fiber of $f'$ is smooth. 
\end{thm'}

We call the morphism $f':(X',B')\to (Y',C')$ a {\it well-prepared birational model} of $f:X\to Y$.
We will not use the so-called weak semi-stable reduction.
A toroidal variety $(X,B)$ is said to be {\it smooth} if $X$ is smooth and $B$ has only normal crossings.
It is {\it quasi-smooth} if there exists a local toric model of each point, which is a quotient of a smooth toric local model by a finite abelian group (some people say that $X$ is a $V$-manifold and $B$ is $V$-normal crossing).

\begin{proof}[{Proof of Theorem \ref{wp}}]Except for condition (4), the statement is the same as  \cite[Theorem 2]{Ka15}. After toroidalization, to make the fibers equidimensional, we do not change the generic fiber when applying \cite[Proposition 4.4]{AK}.   
The condition (4) follows. 

\end{proof}

\medskip

We will pass to a smooth model of $X$.
We may lose the well-preparedness. 

\begin{lem'}\label{modif}
Let $f:(X,P)\to (Y,Q), D, M, L$ be as in \ref{tKsp}\,(1)--(4).
Let $\ga:X'\to X$ be a log-resolution of the pair $(X,P)$.
Write as 
$$
	K_{X'}+D' = \ga^{*}(K_{X}+D)
$$
for a $\BQ$-divisor $D'$ on $X'$, and let $D' = \sum_{j} d'_{j} P'_{j}\ (d'_{j}\in \BQ)$ be the decomposition into irreducible components.
Then, 
$f' =f \circ \ga : X' \to Y, Q, D', M'=\ga^{*}M$ and $L$ satisfy the following (1)--(3):
\begin{itemize}
\item[(1)] $d'_j<1$,
\item[(2)] $\dim_{\mathbb C(\eta)}f'_*\mathcal O_{X'}
(\ulcorner -D'\urcorner)\otimes_{\mathcal O_Y}\mathbb C(\eta)=1$, 
where $\eta$ is the generic point of $Y$, and
\item[(3)] 
$K_{X'}+D'+M' \sim_{\mathbb Q} {f'}^*(K_Y+L)$.
\end{itemize}

Moreover, consider the discriminant part $\Del_{Y}'$ and the moduli part $M_{Y}' = L - \Del_{Y}'$ for $f'$ and $Q$ as in \ref{tKsp}\,(5).
Then, $\Del_{Y}'=\Del_{Y}$ holds and hence $M_{Y}' = M_{Y}$ as a consequence.
\end{lem'}

\begin{proof}
(i) 
We can see that $(X',D')$ is sub-klt as $(X,D)$ is sub-klt, and hence that $D'$ satisfies (1).
We have $D'=\ga^*D - K_{X'/X}=\wtil D_+ - \wtil D_- +E_+ - E_-$,
where $\wtil D_+$ and $\wtil D_-$ are the strict transforms of $D_+$ and $D_-$ respectively, and where $E_+$ and $E_-$ are effective $\ga$-exceptional divisors without common components.
All the coefficients of $\wtil D_+$ and $E_+$ need to be strictly less than 1.
Then $\ru{-D'} = \ru{ \wtil D_-} + \ru{E_-}$.
At a general point $y\in Y$, $X_{y}$ is normal, $X'_{y}$ is smooth, and the effective divisor $\ru{E_-}|_{X'_y}$ is $\ga_y$-exceptional (or $\ru{E_-}|_{X'_y}=0$) with respect to the induced morphism $\ga_y: X'_y \to X_y$. 
Using this, we have 
$1 \le h^0(X'_y, \ru{ \wtil D_-} + \ru{E_-}) \le h^0(X_y, \ru{D_-})=1$.
Thus $f'_*\CO_{X'}(\ru{-D'})$ is of rank 1, that is (2).
We have (3) by construction.

\smallskip

(ii)
Let us see that the discriminant part does not depend on the birational model of $X$. Let $Q_l$ be an irreducible component of $Q$.
For $f:X\to Y$, we set 
$$
	b_{Q_l}=\max \left\{t \in \mathbb Q\, \left|\, 
	\begin{array}{l}  {\text{$(X, D+tf^*{Q_l})$ is sublc over}}\\
	{\text{the generic point of $Q_l$}}. 
	\end{array}\right. \right\}. 
$$

Let us see first that $\delta_{l} =1- b_{Q_l}$. 
By shrinking $Y$ at the generic point of $Q_l$, we may assume that all $P_j$ dominate $Q_l$. 
We consider the pair $(X, D+ (1- \delta_{l})f^*Q_l )$. 
Then, the pair is quasi-smooth, and every coefficient is less than or equal to $1$.
Moreover, there exists some $P_j$ such that the coefficient is $1$. 
Then, it is enough to show that for a quasi-smooth pair $(X,B)$ is sublc if and only if $B \leq 1$ (i.e., every coefficient of $B$ is less than or equal to 1).
But this follows from \cite[Proposition 5.20 (4)]{komo}.

We let
\begin{eqnarray*} 
{f'}^*Q_l &=& 
\sum_{j}w'_{lj}P'_j, \ \text{where}\ w'_{lj}>0, \nonumber\\
\bar{d'}_j&=&\frac{d'_j+w'_{lj}-1}{w'_{lj}}\ 
\text{if}\ f'(P'_j)=Q_l, \nonumber\\
\delta'_{l}&=& \mathrm{max}\{\bar{d'_j}| f'(P'_j)=Q_l\}. \nonumber
\end{eqnarray*}
We would like to show that $\delta_{l}=\delta'_{l}$.
We observed that $ \delta_{l} =1- b_{Q_l}$. 
On the other hand, it holds that  
$ \delta'_{l} =1- \mathrm{lct}_{{f'}^* Q_{lgen}} (X', D')$, where 
$$
\mathrm{lct}_{{f'}^* Q_{lgen}} (X', D') 
=\max \left\{t \in \mathbb Q\, \left|\, 
\begin{array}{l}  {\text{$(X', D'+t {f'}^*{Q_l})$ is sublc over}}\\
{\text{the generic point of $Q_l$}}. 
\end{array}\right. \right\}. 
$$ 
Since such log-canonical thresholds do not change by the crepant pullback, 
we have $b_{Q_l}=  \mathrm{lct}_{{f'}^* Q_{lgen}} (X', D')$. 
That is the desired conclusion. 
\end{proof}

\subsection{Analytic preliminary}

We recall a few basic facts on the Lelong number, which measures the ``analytic''  singularities of plurisubharmonic (psh for short) functions.
We refer to \cite[2.B, 5.B]{Dem} for some properties of psh functions related to algebraic geometry.
We also refer to \cite{Dem} for a general theory of singular Hermitian metrics, multiplier ideal sheaves, and so on.

\begin{dfn'}
Let $\varphi$ be a (quasi-)psh function on an open set $B$ in $\mathbb{C}^n$. 
The {\textit{Lelong number}} $\nu(\varphi, x)$ of $\varphi$ at $x \in B$ is defined by 
$$
	\nu(\varphi, x)=\liminf_{z \to x}\frac{\varphi(z)}{\mathrm{log}|z-x|}.
$$
The Lelong number does not depend on the choice of local coordinates and can, hence, be defined on complex manifolds.
\end{dfn'}

A singular Hermitian metric $h$ on a line bundle $L$ can be written as $h=e^{-\vph}$ with $\vph \in L^{1}_{loc}(U,\BR)$ on every local trivialization $L|_{U} \cong U \times \BC$.
Every such $\vph$ is said to be a weight function of $h$.
If the curvature current of $h$ is semi-positive, then $\vph$ is psh on $U$.
The curvature positivity of $h$ relates to positivity notions in algebraic geometry.
The following result is particularly related to this paper, which is derived from Nadel's vanishing theorem (\cite[Proposition 3.5]{FF}).
See also \cite[Proposition 3.7]{Dem92} for a stronger statement.

\begin{prop'}\label{L0nef}
A line bundle on a smooth projective variety admitting a singular Hermitian metric $h$ with semi-positive curvature and whose local weight functions have vanishing Lelong numbers everywhere is nef. Namely, the intersection number with any curve is semi-positive. 
\end{prop'}

The following is a particular case of \cite[Theorem 2]{Fa}.

\begin{lem'}\label{Fav}
Let $f:U\to V$ be a surjective holomorphic map of complex manifolds, and let $\vph$ be a psh function on $V$.
Let $p\in U$.
Then, $\nu(\vph,f(p))=0$ if and only if $\nu(f^*\vph,p)=0$.
\end{lem'}

Although we do not try to define the Lelong number on singular spaces, we can introduce the zero Lelong number property.
Refer to \cite[\S 1]{Pe} for psh functions on possibly singular complex spaces.

\begin{dfn'}\label{defLe=0}
Let $V$ be a normal complex space, and let $\vph$ be a psh function on $V$.
Let $q\in V$.
The function {\it $\vph$ has zero (or vanishing) Lelong number at $q$}, if there exist a complex manifold $U$, a surjective holomorphic map $f:U\to V$ and a point $p\in f^{-1}(q)$ such that $\nu(f^*\vph,p)=0$.
\end{dfn'}

This property is preserved by an arbitrary pull-back.

\begin{lem'}\label{Le=0}
Let $V$ be a normal complex space, and let $\vph$ be a psh function on $V$.
Let $q\in V$.
Suppose $\vph$ has zero Lelong number at $q$.
Let $g:W\to V$ be a surjective holomorphic map from a normal complex space $W$ and take a point $r\in g^{-1}(q)$.
Then, $g^{*}\vph$ has zero Lelong number at $r$.
In particular, if $W$ is smooth, $\nu(g^*\vph,r)=0$.
\end{lem'}

\begin{proof}
Let $f:U\to V$ and $p\in f^{-1}(q)$ be as in \ref{defLe=0} such that $\nu(f^*\vph,p)=0$.
(We can restrict everything over a relatively compact open neighborhood of $q$.)
By taking the fiber product $W\times_V U$ and a resolution of singularities of it, we can find a complex manifold $X$ with surjective holomorphic maps $f':X\to W$ and $g':X\to U$ such that $g\circ f'=f\circ g'$.
We can also take a point $x\in X$ such that $f'(x)=r$ and $g'(x)=p$.
By \ref{Fav}, the property $\nu(f^*\vph,p)=0$ is equivalent to $\nu((g\circ f')^*\vph,x)=0$, which is saying that $g^{*}\vph$ has zero Lelong number at $r$.
\end{proof}

We will use the following criterion to check the vanishing of a Lelong number.

\begin{lem'}\label{nu0}
Let $U$ be a unit ball in $(\BC^n, z=(z_{1},\ldots,z_{n}))$,
and let $\vph$ and $u$ be quasi-psh functions on $U$ with $\nu(u,0)=n$.
Let $k \in (0,2]$.
Suppose $e^{-\vph} e^{- cu} \in L^1_{loc}(U)$ for any $c \in (0,k)$.
Then $\nu(\vph,0) \le (2-k)n$.
For example, we can conclude $\nu(\vph,0)=0$, if we can take $k=2$ and $u= \log|\prod_{j=1}^n z_j|$.
\end{lem'}

\begin{proof}
Suppose on the contrary that $\nu(\vph,0)>(2-k)n$.
We let $a > 0$ so that $\nu(\vph,0) = (2-k)n + a$, and take $b>0$ such that $b < \min\{k,a/n\}$.
Then $\vph+(k-b)u$ is quasi-psh on $U$ and
$\nu(\vph+(k-b)u,0) \ge \nu(\vph,0) +(k-b)\nu(u,0)
=2n+a-bn>2n$.
Then $e^{-\vph} e^{- (k-b)u} \not\in L^1_{loc}(U)$ (\cite[5.6]{Dem}).
This is a contradiction.
\end{proof}

Psh functions with vanishing Lelong numbers do not affect multiplier ideals.

\begin{lem'}\label{open}
Let $U$ be a unit ball in $\BC^n$,
and let $\vph$ and $u$ be quasi-psh functions on $U$.
Suppose $\nu(\vph,0)=0$.
Then $\CI(e^{-\vph} e^{-u}) = \CI(e^{-u})$ on a neighborhood of $0$.
\end{lem'}

Here, $\CI(e^{-u})\subset \CO_{U}$ is a multiplier ideal sheaf.
It is the sheaf of germs of holomorphic functions $f$ which is $L^{2}$ with respect to $e^{-u}dv$, where $dv$ is the standard Lebesgue measure on $U$, namely $|f|^{2}e^{-u} \in L^{1}_{loc}(U)$.
The multiplier ideal sheaf is coherent if $u$ is quasi-psh (\cite[5.7]{Dem}).

\begin{proof}
$\CI(e^{-\vph} e^{-u}) \subset \CI(e^{-u})$ is clear.
We take $f\in \CI(e^{-u})_{0}$.
We can take a relatively compact open neighborhood $V$ of $0$ in $U$ such that $\int_{V} |f|^{2} e^{-u} < \infty$.
Here, we take the Lebesgue measure as a volume form on $U$.
We take $p,q\in \BR_{>1}$ with $1/p+1/q=1$ such that $\CI(e^{-qu}) = \CI(e^{-u})$ on $V$.
The last property is the so-called strong openness property of multiplier ideals \cite{GZ} (see also \cite{Be}). 
As $\nu(p\vph, 0)=0$, we can take a neighborhood $W\subset V$ of $0$ such that $e^{-p\vph} \in L^{1}(W)$ by Skoda's lemma (\cite[5.6]{Dem}).
Then, by H\"older's inequality,
$$
	\int_{W} |f|^{2} e^{-\vph} e^{-u}
	\le \Big( \int_{W} |f|^{2} e^{-p\vph} \Big)^{1/p} 
	     \Big( \int_{W} |f|^{2} e^{-q u} \Big)^{1/q}. 
$$
Thus, we have $f \in \CI(e^{-\vph} e^{-u})_{0}$.
If $u$ has only analytic singularities, the openness property is an elementary property.
\end{proof}

The next lemma is trivial.

\begin{lem'}\label{zint}
Let $a_j, b_j \in \BR \ (j=1,2,\ldots,n)$ with $a_j<2$ for any $j$.
Then
$$
	\prod_{j=1}^n \frac{ (-\log |z_j|)^{b_j} }{ |z_j|^{a_j} } \in L^1_{loc}(\BC^n).
$$
\end{lem'}

\section{Fiber integral}

\subsection{Basic estimate: Smooth total space}

Let everything be as in \ref{tKsp}, and suppose $X$ is smooth.
We let $\dim X=n+m$ and $\dim Y = m$. 
We explain a key object: a fiber integral, which will appear repeatedly in this paper (and can be formulated in various other situations).

Let us recall a convention, which is commonly used in complex algebraic geometry.
Let $\Del$ be an effective $\BQ$-divisor, say on $X$ in general.
We take a positive integer $k$ such that $k\Del$ is a divisor with integral coefficients, a section $s_{k\Del} \in H^{0}(X,k\Del)$ whose divisor is $k\Del$, and a Hermitian metric $h_{k\Del}$ on the line bundle $k\Del$.
Then, $s_{\Del}:= (s_{k\Del})^{1/k}$ is a root of section (or a multi-valued holomorphic section), and $h_{\Del} := (h_{k\Del} )^{1/k}$ is a Hermitian metric on the $\BQ$-line bundle $\Del$ in a somewhat formal sense.  
But, $|s_{\Del}|_{h_{\Del}} := (|s_{k\Del}|_{h_{k\Del}})^{1/k}
\in C^{0}(X,\BR_{\ge 0})$ is a function in the usual sense.


\begin{no'}\label{intF}
We let
$$
	E = \ru{D_{-}} - K_{X/Y},
$$
which is a line bundle on $X$.
We take any Hermitian metric $h$ on $E$.
We take a root of section $s_{\Del_{D}}$ whose divisor is $\Del_{D}$, and take a Hermitian metric $h_{\Del_{D}}$ on the $\BQ$-line bundle $\Del_{D}$.
Let $\vph \in L^{1}_{loc}(X,\BR)$ be a quasi-psh function.
We finally let 
$$
	g =  h \ \frac{ e^{-\vph} }{ |s_{\Del_{D}}|_{h_{\Del_{D}}}^{2} }
$$ 
be a singular Hermitian metric on $E$.
We take a section 
$$
	\sg=s_{\ru{D_{-}}} \in H^{0}(X,\ru{D_{-}}) = H^{0}(X, K_{X/Y}+E)
$$ 
whose divisor is $\ru{D_{-}}= \ru{-D}$.
We regard $\sg$ as an $E$-valued relative holomorphic $n$-form.
We then consider the associated relative (possibly singular) volume form $\w$ and the fiber ($L^{2}$-)integral of $\sg$ with respect to $g$:
$$
	\w = g \ c_{n}\sg\wed\ol\sg, \ {\rm and } \ 
	F_{\w}(y)=\int_{X_{y}} \w \ \in [0,+\infty]
$$
for $y \in Y\sm Q$.
Here $c_{n}=\ai^{n^{2}}$ so that the pairing becomes semi-positive on each fiber.
\qed\end{no'}

For $y \in Y\sm Q$, $F_{\w}(y) < +\infty$ if and only if 
$\sg|_{X_{y}} \in H^{0}(X_{y}, K_{X_{y}}\ot E|_{X_{y}} \ot \CI(g|_{X_{y}}))$,
where $\CI(g|_{X_{y}})$ is the multiplier ideal sheaf of $g|_{X_{y}}$.
We understand $F_{\w}(y)=+\infty$ if $\vph|_{X_{y}}\equiv -\infty$ on a fiber for example.
We know, in general, $F_{\w}(y)$ is lower-semi-continuous on which $f$ is smooth, \cite[22.5]{HPS}.

We are interested in the behavior of $F_{\w}(y)$ as $y$ approaches to $Q$.
Here, we are only concerned about the singularity type of $g$, not about the curvature property of $g$, e.g.~ semi-positivity.
The choices of smooth metrics $h, h_{\Del_{D}}$ also do not affect the following conclusion.

\begin{thm'}\label{L1}
Suppose $X$ is smooth and $\Del_{Y}=0$ in \ref{tKsp}, and use the notations in \ref{intF}.
Suppose the Lelong number $\nu(\vph,x)=0$ for every $x \in X$.
Let $y_{0}\in Y$ and let $V \subset (\BC^m,t=(t_1,\ldots,t_m))$ be a local coordinate centered at a point $y_0\in Y$.
Then

(1) Upper estimate.
$F_{\w} /|\prod_{j=1}^m t_j|^{c} \in L^1_{loc}(V,0)$ for any $c\in (0,2)$.
In particular, if $-\log F_{\w}$ is quasi-psh on a neighborhood of $y_{0}$,
one has $\nu(-\log F_{\w}, y_{0})=0$.

(2) Lower estimate.
There exist positive constants $\del(<1)$ and $c_{1}$ such that, for any $t\in V \sm Q$ with $|t|<\del$, $F_{\w}(t) \ge c_{1}$ holds.
\end{thm'}

\begin{proof}
(1)
{\it Step 1. Preliminary discussion.}
We take $c\in (0,2)$.
We put $T=\prod_{j=}^{m} t_{j}$, and $|s_{\Del_{D}}| = |s_{\Del_{D}}|_{h_{\Del_{D}}}$ for short.
We note $h c_{n}\sg\wed\ol\sg |dt|^{2}$ is a $C^{\infty}$ semi-positive volume form on $X_{V}=f^{-1}(V)$, where $dt=dt_{1}\wed\cdots\wed dt_{m}$ and $|dt|^{2}= c_{m}dt\wed d\ol t$.
Then
$$
	\int_{V} F_{\w} \ \frac{ |dt|^{2} }{ |T|^{c} }
	= \int_{X_{V}}	  \frac{ e^{-\vph} }{ |s_{\Del_{D}}|^{2} |T|^{c} }
	h  c_{n}\sg\wed\ol\sg |dt|^{2}.
$$
We take an open covering $\CU = \{U_{\lam} \}_{\lam}$ of $X$, which is fine enough and locally finite, and take a partition of unity $\{\rho_{\lam}\}_{\lam}$ subordinate to $\CU$.
We can suppose $\ol {X_{V}} \cap U_{\lam} \ne \emptyset$ for a finite number of $\lam$.
Then
$$
	\int_{V} F_{\w} \ \frac{ |dt|^{2} }{ |T|^{c} }
	= \sum_{\lam} \int_{U_{\lam}\cap X_{V}} \rho_{\lam} 
	 \frac{ e^{-\vph} }{ |s_{\Del_{D}}|^{2} |T|^{c} }
	 h c_{n}\sg\wed\ol\sg |dt|^{2}.
$$
Thus, it is enough to evaluate an integral on each local coordinate $U$ of $X$.
We let $X_{t}=f^{-1}(t)$ for $t \in V$.

\smallskip

{\it Step 2. Local coordinate.}
At each $x_{0}\in X_{0}$, we can find local coordinates $U\subset (\BC^{n+m}, x=(x_{1},\ldots,x_{n+m})), V\subset (\BC^{m}, t=(t_{1},\ldots,t_{m}))$
(we suppose $|x_{i}| \le 1$ for all $i$, $|t_{j}| \le 1$ for all $j$), 
such that $f|_{U}:U\to V$ is given by $t=f(x), t_{j}=f_{j}(x), j=1,\ldots,m$, with
$$
	t_{1}=\prod_{j=1}^{\ell_{1}} x_{j}^{a_{j}}, \ 
	t_{2}=\prod_{j=\ell_{1}+1}^{\ell_{2}} x_{j}^{a_{j}}, \ \ldots, \ 
	t_{m}=\prod_{j=\ell_{m-1}+1}^{\ell_{m}} x_{j}^{a_{j}} 
$$
for some $0=\ell_{0}<\ell_{1}<\ldots < \ell_{m}\le n+m$ and $a_j\in \BZ_{>0}$.
(The coordinates $x,t$ depend on $x_{0}$, but)
Moreover $Q\cap V\subset \{t_{1}\cdots t_{m}=0\}$ and $P\cap U \subset \{x_{1}\cdots x_{n+m}=0\}$. 
This $V\subset (\BC^{m},t)$ may be different from the one we took in the statement of \ref{L1}.
However, it does not matter what our purpose is.
We can write as $\ell_{i}=\ell_{i-1}+n_{i}+1$ with $n_{i}\in \BZ_{\ge 0}$ and $\sum_{i=1}^{m+1} n_{i}=n$, where $n_{m+1}:=n+m-\ell_{m}$.
We denote by $y_{j}=x_{\ell_{j}}, b_{j}=a_{\ell_{j}}$ for $j=1,\ldots,m$, and by
\begin{equation*} 
\begin{aligned}
&\bfx_{j}=(x_{\ell_{j-1}+1},\ldots,x_{\ell_{j-1}+n_{j}}), \
\bfa_{j}=(a_{\ell_{j-1}+1},\ldots,a_{\ell_{j-1}+n_{j}}), \\
&\bfx_{m+1}=(x_{\ell_{m}+1},\ldots,x_{n+m}).
\end{aligned}
\end{equation*} 
We will use a multi-index convention. 
For example $t_{j}=\bfx_{j}^{\bfa_{j}} y_{j}^{b_{j}}$.

\smallskip

{\it Step 3. Integrand.}
We may suppose that every ($\BQ$-)line bundle appearing in our situation is trivialized on a neighborhood of $\ol U$ or of $\ol V$.
We take a local frame $dx=\wed_{i=1}^{n+m} dx_{i}$ of $K_{X}|_{U}$ and $dt=\wed_{j=1}^{m} dt_{j}$ (as before) of $K_{Y}|_{V}$ respectively.
Using relations $\frac{dt_{j}}{t_{j}} 
= \sum_{i=\ell_{j-1}+1}^{\ell_{j-1}+n_{j}} a_{i} \frac{dx_i }{ x_{i} } + b_{j}\frac{dy_{j}}{y_{j}}$, we have
$\wed_{i=\ell_{j-1}+1}^{\ell_{j}} dx_{i}
=  \frac{y_{j}}{b_{j}t_{j}} \big( \wed_{i=\ell_{j-1}+1}^{\ell_{j-1}+n_{j}} dx_{i} \big) \wed dt_{j}$, in other term $d\bfx_{j}\wed dy_{j}=  \frac{y_{j}}{b_{j}t_{j}} d\bfx_{j} \wed dt_{j}$.
Then, a local frame $\frac{dx}{dt}$ of $K_{X/Y}$ on $U$ can be expressed as
$$
	\frac{dx}{dt} 
	= \bigwedge_{j=1}^{m} \Big(\frac{y_{j}}{b_{j}t_{j}}  d\bfx_{j}\Big) \wed d\bfx_{m+1},
$$
which satisfies $\frac{dx}{dt} \wed f^{*}dt  = \pm dx$.
We also let $|dx|^{2} = c_{n+m}dx\wed d\ol x$, and
$\big| \frac{dx}{dt}\big| = c_{n} \frac{dx}{dt} \wed \ol {\frac{dx}{dt}}$. 

Regarding $\sg \in H^{0}(X,K_{X/Y}+E)= {\rm Hom}_{\CO_{X}}(f^{*}K_{Y},K_{X}+E)$, we have $\sg(f^{*}dt)|_{U} = \sg_{U}dx  = \pm \sg_U \frac{dx}{dt} \wed f^{*}dt$ for some $\sg_{U}\in H^{0}(U,\CO_{X}) (\cong H^{0}(U,E))$, and hence
$$
	\sg|_{X_{t}} = \pm \sg_U \frac{dx}{dt} \Big|_{X_{t}}
$$ 
for $t\in V\sm Q$.
The metric $h$ can be seen as a function in $C^{\infty}(\ol U, \BR_{>0})$.
Thus we have 
$$
	h c_{n}\sg\wed\ol\sg |dt|^{2} = A_{U}|\sg_{U} |^{2} |dx|^{2}
$$ 
on $U$ for some $A_{U} \in C^{\infty}(\ol U, \BR_{>0})$.
We finally take a cut-off function $\rho_{U} \in C^{\infty}_{c}(U, \BR_{\ge 0})$.
Then, in view of Step 1, we need to estimate
$$
	I_{U} = \int_{U} \rho_{U} 
	  e^{-\vph} \frac{ A_{U} |\sg_{U} |^{2} }{ |T|^{c} |s_{\Del_{D}}|^{2} }
	 |dx|^{2}. 
$$

\smallskip

{\it Step 4. Reduction to an un-twisted case.}
We introduce auxiliary quantities.
Recall $f^{*}Q_{j}= \sum_{i}a_{ji}P_{i}$ for each $j$.
We set $c_{j} = \max_{i}\{d_{i}/a_{ji} \ {\rm if } \ f(P_{i}) = Q_{j} \}$, and
$\wtil c = \max_{j} c_{j} \in \BR$.
We take $\ep > 0$ such that $c(1+\ep)+2\wtil c \ep<2$.
A simpler condition $c(1+\ep)+2d_{i}\ep < 2$ for any coefficient $d_{i}$ of $P_{i}$ also works (which in fact implies $c(1+\ep)+2\wtil c \ep<2$).
We also require that $d_{i}(1+\ep) < 1$ holds for any $i$ with $d_{i}<1$.
We then take $q=1+\ep$ and $p \gg 1$ such that
$\frac1p+\frac1q=1$. 

By H\"older's inequality, we have
$$
	I_{U} 
	\le \big\| \rho_{U}^{1/p} e^{-\vph} \big\|_{L^{p}(U, |dx|^{2})} 
		\cdot
	\Big\| \rho_{U}^{1/q} \frac{ A_{U} |\sg_{U} |^{2} }{ |T|^{c}  |s_{\Del_{D}}|^{2} }
	 \Big\|_{L^{q}(U, |dx|^{2})}. 
$$
The first term on the right hand side is finite:
$$
	(J_{p}:=) \ \big\| \rho_{U}^{1/p} e^{-\vph} \big\|_{L^{p}(U, |dx|^{2})}^{p}
	= \int_{U} \rho_{U} e^{-p\vph} |dx|^{2} < \infty,
$$
by virtue of our assumption $\nu(\vph,x)=0$ for any $x\in X$.
We next consider 
$$
	J_{q} := \int_{U} \rho_{U} 
	\frac{ A_{U}^{q} |\sg_{U} |^{2q} }{ |T|^{cq} |s_{\Del_{D}}|^{2q} } |dx|^{2}.
$$	
By Fubini, we have
$$
	J_{q}
	= \int_{t\in V} \frac{ |dt|^{2} }{ |T|^{cq} }
	\int_{X_{t}\cap U} \rho_{U} A_{U}^{q} 
	\frac{ |\sg_{U} |^{2q} }{ |s_{\Del_{D}}|^{2q} } \Big|\frac{dx}{dt} \Big|^{2}.	
$$	
We will show in the next step that the fiber integral is estimated as
$$
	\int_{X_{t}\cap U} \rho_{U} A_{U}^{q} 
	\frac{ |\sg_{U} |^{2q} }{ |s_{\Del_{D}}|^{2q} } \Big|\frac{dx}{dt} \Big|^{2}
	\le
	\frac{ A }{  |T|^{2\wtil c \ep} } \prod_{j=1}^{m} (-\log |t_{j}|)^{n}
$$
for any $t \in V\sm Q$, where $A>0$ is a constant independent of $t$.
This estimate deduces that
$$
	J_{q} 
	\le A \int_{t \in V} 
	\frac{ \prod_{j=1}^{m} (-\log |t_{j}|)^{n} }{ |T|^{cq+2\wtil c \ep} }\ |dt|^{2}
	<\infty.
$$
The last integrability follows from $cq+2\wtil c \ep <2$ and \ref{zint}.
We have $I_{U} < \infty$ as a conclusion.

The statement $\nu(-\log F_{\w},y)=0$ is a consequence of this integrability combined with \ref{nu0}.

\smallskip

{\it Step 5. Estimate of fiber integral.}
We can suppose $V\cap Q= \{t_{1}\cdots t_{k}=0\}$ for some $k$ with $0 \le k \le m$, and $f$ is smooth over $V \sm Q$ and singular over $V\cap Q$ (cf.\ \cite[2.2\,(2)]{Ta}).
We can let $Q_{j}\cap V=\{t_{j}=0\}$ for $1 \le j \le k$, and $P_{i}\cap U=\{x_{i}=0\}$ for $1\le i \le \ell_{k}$ (such that $P_{i}$ is $f$-vertical and appeared in $D$).
For a notational simplicity, we suppose $k=m$ (and hence $\ell_{k}= \ell_{m}$).
Other cases $0 \le k\le m-1$ can be deduced from the case $k=m$.
In particular $P_{i}$ is $f$-vertical if and only if $1\le i \le \ell_{m}$.
We write as 
$$
	D^{h}={\rm div} \Big(\prod_{i > \ell_{m}, d_{i}<1} x_{i}^{d_{i}} \Big), \
	D^{v}_{-}={\rm div} \Big(\prod_{ i \le \ell_{m}, -d_{i}>0} x_{i}^{-d_{i}} \Big), \
	D^{v}_{+} = {\rm div }\, (1) = 0.
$$
Here $D_{+}^{v}=0$ by our assumption:\ $f$ is equi-dimensional and $\Del_{Y}=0$ (cf.\ \cite[2.2\,(3)]{Ta}).

We estimate the integrand of the fiber integral in $J_{q}$.
We have
\begin{equation*} 
\begin{aligned}
	|\sg_{U}|^{2} 
	& = A _{1} \cdot \prod_{i, -d_{i}>0} |x_{i}|^{2\ru{-d_{i}}}, 
	\\
	|s_{\Del_{D}}|^{2}
	& = A _{2} \cdot \prod_{i>\ell_{m}, 0<d_{i}<1} |x_{i}|^{2d_{i}} 
	 \prod_{i, -d_{i}>0} |x_{i}|^{ 2(\ru{-d_{i}} + d_{i}) }  
\end{aligned}
\end{equation*} 
on $U$, where $A_{1}, A_{2} \in C^{\infty}(\ol U, \BR_{>0})$.
Then 
$$
	\frac{ |\sg_{U} |^{2q} }{ |s_{\Del_{D}}|^{2q} } \Big|\frac{dx}{dt} \Big|^{2}
	= A_{3} \bigwedge_{j=1}^{m+1} W_{j} 
$$
holds with $A_{3}= A_{1}^{q} A_{2}^{-q} \prod_{j=1}^m b_j^{-2} \in C^{\infty}(\ol U, \BR_{>0})$, and
\begin{equation*} 
\begin{aligned}
	& W_{j} = ( |y_j|^{2})^{-d_{\ell_j}q+1-a_{\ell_j}} 
	\prod_{i=\ell_{j-1}+1}^{\ell_{j-1}+n_{j}} (|x_i|^2)^{-d_iq-a_i}
	\  c_{n_{j}} d\bfx_{j}\wed d\ol \bfx_{j} \ \ (j= 1, \ldots, m),
\\ 
	& W_{m+1} = \frac{ \prod_{i>\ell_m, d_{i} <0 } |x_{i}|^{-2d_{i}q} }
	{ \prod_{i > \ell_m, 0 < d_{i} <1 } |x_{i}|^{2d_{i}q} }\
	c_{n_{m+1}} d\bfx_{m+1}\wed d\ol \bfx_{m+1}.
\end{aligned}
\end{equation*} 

\smallskip

We now start to estimate.
Let us look at $W_{1}$ for example, and show
$$
	W_{1} \le \frac{1}{|t_{1}|^{2c_{1}\ep}}
 			 \frac{ c_{n_1} d\bfx_1 \wed d\ol \bfx_1 }{ | \bfx_{1} |^{2} }
$$
as $(n_{1},n_{1})$-form on $U$, where $| \bfx_{1} |^{2} = \prod_{i=1}^{n_{1}} |x_{i}|^{2}$ (and $| \bfx_{j} |^{2} = \prod_{i=\ell_{j-1}+1}^{\ell_{j-1}+n_{j}} |x_{i}|^{2}$ in general). 

(i) 
For $t_{1}=\prod_{j=1}^{\ell_{1}} x_{j}^{a_{j}}$, we have
$|t_{1}|^{c_{1}} = \prod_{j=1}^{\ell_{1}} | x_{j}^{a_{j}} |^{c_{1}}
\le \prod_{j=1}^{\ell_{1}} |x_{j}^{a_{j}} |^{d_{j}/a_{j}}$ (by definition of $c_{1}$ and as $|x_{j}|\le 1$).
Thus, $|t_{1}|^{c_{1}\ep } \le \prod_{j=1}^{\ell_{1}} |x_{j} |^{d_{j}\ep}$.

(ii)
We note, as a consequence of $\Del_{Y}=0$, that for every $j\in \{1,\ldots,m\}$ and $i\in \{\ell_{j-1}+1,\ldots,\ell_{j}\},  -d_{i}-a_{i}\ge -1$ holds, in particular 
$-d_{\ell_j}+1-a_{\ell_j}\ge 0$ holds.
Recalling $q=1+\ep$, we use it as $-d_{\ell_j}q+1-a_{\ell_j} \ge -d_{\ell_j}\ep$ and $-d_{i}q-a_{i}\ge -1-d_{i}\ep$.

(iii) 
The ``coefficient function'' of $W_{1}$ can be estimated as
$( |y_1|^{2})^{-d_{\ell_1}q+1-a_{\ell_1}} \prod_{i=1}^{n_{1}} (|x_i|^2)^{-d_iq-a_i}$
$\le_{{\rm (ii)}}$ 
$( |y_1|^{2})^{-d_{\ell_1}\ep } \prod_{i=1}^{n_{1}} |x_{i}|^{-2} (|x_i|^2)^{-d_i\ep}
\le_{{\rm (i)}} 
|t_{1}|^{-2c_{1}\ep} | \bfx_{1} |^{-2}$.
This gives the desired estimate of $W_{1}$.

\vskip2mm
We have a similar estimate for $W_{j} \ (j=2,\ldots,m)$.
We note $\prod_{j=1}^{m} |t_{j}|^{-2c_{j}\ep} \le |T|^{-2\wtil c \ep}$.
For $W_{m+1}$, we simply estimate it as
$$
	W_{m+1} \le \frac{ c_{n_{m+1}} d\bfx_{m+1}\wed d\ol \bfx_{m+1}  }
	{ \prod_{i > \ell_m, 0 < d_{i} <1 } |x_{i}|^{2d_{i}q} }.
$$
We introduce a (possibly singular) semi-positive $(n,n)$-form
$$
 \psi =
 	\bigwedge_{j=1}^m \frac{ c_{n_j} d\bfx_j \wed d\ol \bfx_j }
				{ | \bfx_{j} |^{2} }
	\wed \frac{ c_{n_{m+1}} d\bfx_{m+1}\wed d\ol \bfx_{m+1} }
 				{ \prod_{i > \ell_m, 0 < d_{i} <1 } |x_{i}|^{2d_{i} q} }
$$
on $(\BC^{n+m}, x)$.
We then have a constant $A_{4}>0$ (independent of $t$) such that
$$
	\rho_{U} A_{U}^{q} 
	\frac{ |\sg_{U} |^{2q} }{ |s_{\Del_{D}}|^{2q} } \Big|\frac{dx}{dt} \Big|^{2}
	= 
	\rho_{U} A_{U}^{q} A_{3} \bigwedge_{j=1}^{m+1} W_{j} 
	\le \frac{A_{4}}{|T|^{2\wtil c\ep}} \ \psi
$$ 
holds on $U$.

As $d_{i}q<1$ in the expression of $\psi$, 
we can directly apply the way of computation in \cite[p.1733, Step 4]{Ta} to conclude that
$$
	\int_{X_{t}\cap U} \psi
	\le
	A_{5} \prod_{j=1}^{m} (-\log |t_{j}|)^{n}
$$
holds for any $t\in V \sm Q$, where $A_{5}>0$ is a constant independent of $t$.
Thus we have
$$	\int_{X_{t}\cap U} \rho_{U} A_{U}^{q}
	 \frac{  |\sg_{U} |^{2q} }{ |s_{\Del_{D}}|^{2q} }
	 \Big|\frac{dx}{dt} \Big|^{2}
	\le \frac{A_{4}A_{5}}{|T|^{2\wtil c\ep}} \prod_{j=1}^{m} (-\log |t_{j}|)^{n}
$$
for any $t \in V \sm Q$.
This is what we wanted.

\smallskip

(2) 
Lower estimate.
As $\vph$ is quasi-psh, the function $e^{-\vph}$ has a positive lower bound on every compact set of $X$.
In particular there exists a constant $c_{0}>0$ such that $e^{-\vph}>c_{0}$ on a neighborhood of $X_{y_{0}}$.
(We merely use the existence of a positive lower bound of $e^{-\vph}$, and we do not use the Lelong number condition of $\vph$.)
The integrand is
$$
	\w 
	= h \frac{ e^{-\vph} }{ |s_{\Del_{D}}|^{2} } c_{n}\sg\wed\ol\sg
	\ge h \frac{ c_{0} }{ |s_{\Del_{D}}|^{2} } c_{n}\sg\wed\ol\sg.
$$
The last term is ``an un-twisted case''.
We then just follow the argument in \cite[p.1735, Step 5]{Ta}.
\end{proof}

\begin{rem'}\label{local ver}
Our argument in \ref{L1} is local over $Y$.
It can be generalized as follows.

Let $X$ and $Y$ be complex manifolds, $f : X \to Y$ be a proper surjective holomorphic map with connected fibers, which satisfies the conditions (1)--(3) in \ref{tKsp} in an appropriate sense for ($\BQ$-)divisors $P, Q$ and $D$.
(The property (4), $M, M_{Y}$ and $L$ are irrelevant.)
Suppose $\Del_{Y}=0$ in \ref{tKsp}\,(5).
Take $\vph \in L^{1}_{loc}(X,\BR)$, $s_{\Del_{D}}$,
$\sg=s_{\ru{D_{-}}} \in H^{0}(X,\ru{D_{-}})$, and $h, h_{\Del_{D}}$, and construct a singular Hermitian metric $g$ as in \ref{intF}.
Put  
$$
	F_{\w}(y)=\int_{X_{y}} g\,c_{n}\sg\wed\ol\sg \ \in [0,+\infty]
$$
for $y \in Y\sm Q$.
Let $y_{0}\in Y$ and let $V \subset (\BC^m,t=(t_1,\ldots,t_m))$ be a local coordinate centered at a point $y_0\in Y$.

(1) Suppose $\vph$ is quasi-psh on a neighborhood of $X_{y_{0}}$ and the Lelong number $\nu(\vph,x)=0$ for every $x \in X_{y_{0}}$.
Then, $F_{\w} /|\prod_{j=1}^m t_j|^{c} \in L^1_{loc}(V,0)$ for any $c\in (0,2)$.

(2) Suppose $\vph$ is locally bounded from above on a neighborhood of $X_{y_{0}}$.
Then, there exist positive constants $\del(<1)$ and $c_{1}$ such that, for any $t\in V \sm Q$ with $|t|<\del$, $F_{\w}(t) \ge c_{1}$ holds.
\qed\end{rem'}

\subsection{Quasi-smooth total space}

We explain fiber integrals in the orbi-fold setting. 
We refer to \cite{CM} for a basic discussions on orbi-folds;
\cite[\S 2, \S 3]{CM} gives a nice and concise account on (complex) orbi-folds, orbi-line bundles and so on.

\begin{no'}[Fiber integral]\label{intForb}
Let $X$ be a complex orbi-fold, $Y$ be a complex manifold, and let $f:X\to Y$ be a proper surjective holomorphic map with connected fibers.
We let $\dim X=n+m$ and $\dim Y = m$. 
We regard $\w_{X}=\CO_{X}(K_{X})$ as a reflexive sheaf of rank 1 and an orbi-line bundle.

(1) 
Let $E$ be an orbi-line bundle on $X$, and suppose $E$ is a $\BQ$-line bundle.
Suppose there is a section 
$$
	\sg \in H^{0}(X, K_{X/Y}+E).
$$ 
To be more precise, $K_{X/Y}+E$ means the double dual of the tensor product of orbi-line bundles $K_{X/Y}$ and $E$.
We regard $\sg$ as an orbi-line bundle $E$-valued relative holomorphic $n$-form on $X$ in the following sense.
For every $x \in X$, we can take an open neighborhood $U \subset X$ of $x$ such that there is an open neighborhood $\tU$ of $(\BC^{n+m},0)$ and a finite group $G$, which is ``small'', acting on $\tU$ and $U$ is the quotient $U=\tU/G$.
Let $\pi:\tU \to U=\tU/G$ be the quotient map.
We have $K_{\tU}=\pi^{*}K_{U}$ as a $\BQ$-divisor since there is no codimension 1 ramifications of $\pi$, and a line bundle $\pi^{*}E$ on $\tU$.
We consider
$$
	\pi^{*}\sg \in H^{0}(\tU, K_{\tU/Y}+\pi^{*}E),
$$ 
which can be seen as a (usual) $\pi^{*}E$-valued relative holomorphic $n$-form on $\tU$.
We note that $(\pi^{*}\sg)^{p} = \pi^{*}(\sg^{p})$ holds for each positive integer $p$.
The equality is valid over $\pi^{-1}(U_{reg})$ initially, and continue to hold on $\tU$.

(2) 
Let $g$ be a singular Hermitian metric on the $\BQ$-line bundle $E$ such that every local weight function $-\log g$ is quasi-psh.
We define a (possibly singular) relative volume form 
$$
	\w = g\,c_{n}\sg\wed\ol\sg
$$
on $X$ as follows.
In the first place, $\w = g c_{n}\sg\wed\ol\sg$ is a relative $(n,n)$-form on $X_{reg}$, where $c_{n}=\ai^{n^{2}}$.
For each orbi-fold local chart $\pi:\tU \to U=\tU/G$, $\pi^{*}\w$, defined on $\pi^{-1}(U_{reg})$, extends to $(\pi^{*}g) c_{n} \pi^{*}\sg \wed \ol{\pi^{*}\sg}$ on $\tU$, where $\pi^{*}g$ is a $G$-invariant singular Hermitian metric of $\pi^{*}E$ and $\pi^{*}\sg$ is a $G$-invariant $\pi^{*}E$-valued relative (with respect to $f\circ\pi$) holomorphic $n$-form on $\tU$.
We still denote $\pi^{*}\w= (\pi^{*}g) c_{n} \pi^{*}\sg \wed \ol{\pi^{*}\sg}$ on $\tU$.

(2')
Let $(V,t=(t_1,\ldots,t_m))$ be  a local coordinate of $Y$ and put $X_V=f^{-1}(V)$.
We trivialize $K_Y$ on $V$ by $dt=dt_1\wed\ldots\wed dt_m$.
We take a volume form $|dt|^2 =c_m dt \wed \ol{dt}$ on $V$.
We then understand $\w \wed f^{*}|dt|^{2}$ as a possibly singular semi-positive $(n+m,n+m)$-form (a positive measure) on $X_V$.
It is initially defined on $X_V \cap X_{reg}$ as a positive measure.
On each orbi-fold local chart $\pi:\tU \to U=\tU/G$ such that $U \subset X_V$, we have
$\pi^{*}\sg \in H^{0}(\tU, K_{\tU/Y}+\pi^{*}E)={\rm Hom}_{\CO_{\tU}}((f\circ\pi)^*K_Y, K_{\tU}+\pi^*E)$.
Then we have 
$(\pi^{*}\sg)((f\circ\pi)^{*}dt) \in H^{0}(\tU, K_{\tU}+\pi^{*}E)$, and have a positive measure $(\pi^{*}g) c_{n+m} (\pi^{*}\sg)((f\circ\pi)^{*}dt) \wed \ol{ (\pi^{*}\sg)((f\circ\pi)^{*}dt) }$ on $\tU$ which is $G$-invariant.
The pull-back $\pi^{*}(\w \wed f^{*}|dt|^{2})$ over $U_{reg}$ extends to
$(\pi^{*}g) c_{n+m} (\pi^{*}\sg)((f\circ\pi)^{*}dt) \wed \ol{ (\pi^{*}\sg)((f\circ\pi)^{*}dt) }$.

Regarding $\w \wed f^{*}|dt|^{2}$ as an $(n+m,n+m)$-form on $X_{V}$, we can consider its tensor power $(\w \wed f^{*}|dt|^{2})^{p}$ for every $p\in \BZ_{>0}$. 
We can not integrate it if $p>1$, but we can compare such powers of forms pointwise on $X_{V}$.

(3) 
Denote $Y_{f\,\text{q-sm}}=\{y\in Y;\ \text{$f$ is quasi-smooth over $y$}\}$.
Here $f$ is quasi-smooth at $x\in X$, if the orbi-fold tangent map $df:T_{X,x}\to f^*T_{Y,f(x)}$ is surjective, equivalently $d(f\circ\pi):T_{\tU,\wtil x}\to (f\circ\pi)^*T_{Y,f(x)}$ is surjective for a(ny) $\wtil x\in \tU$ such that $\pi(\wtil x)=x$, where $\pi:\tU\to U/G$ is an orbi-fold local coordinate around $x$.
We define a function $F_{\w}:Y_{f\,\text{q-sm}}\to [0,+\infty]$ given by the fiberwise ($L^{2}$-)integral of $\sg$ with respect to $g$:
$$
	F_{\w}(y)=\int_{X_{y}} \w \ \in [0,+\infty].
$$
We take an open covering $\{U_{k} \}$ of $X$, which is locally finite and each $U_{k}$ is a quotient $\pi_{k}:\tU_{k} \to U_{k}= \tU_{k}/G_{k}$ as above.
We take a partition of unity $\{\rho_{k}\}$ subordinate to $\{U_{k} \}$.
We then define, for every $y\in Y_{f\,\text{q-sm}}$,
\begin{equation*} 
\begin{aligned}
	F_{\w}(y)
	& := \int_{X_{y}} \w 
	=\sum_{k} \int_{X_{y} \cap U_{k} } \rho_{k} \w, \\ 
	\int_{X_{y} \cap U_{k}} \rho_{k}\w 
	& := \frac1{|G_{k}|} \int_{ (f\circ \pi_{k})^{-1}(y) } \pi_{k}^{*} (\rho_{k}\w)
	\ \in [0,+\infty].
\end{aligned}
\end{equation*} 

(3')
Let $(V,t=(t_1,\ldots,t_m))$ be  a local coordinate of $Y$ and put $X_V=f^{-1}(V)$.
We have an $(n+m,n+m)$-form $\w \wed f^{*}|dt|^{2}$ on $X_{V}$.
Then the push-forward is 
$$
	f_{*}(\w \wed f^{*}|dt|^{2}) = F_{\w} |dt|^{2}
$$
on $V\cap Y_{f\,\text{q-sm}}$.
\qed
\end{no'}

In our main concern, we take those things as follows (in \ref{L1} and \ref{L1orb}, for example).
We will also use a version that is localized over the base $Y$ (in the proof of \ref{can sHm}, for example). 

\begin{no'}\label{intForb2}
Let everything be as in \ref{tKsp}.
We then take objects $E, \sg, \ldots$ in \ref{intForb} as follows.

(1)
We take an orbi-line bundle
$$
	E = \ru{D_{-}} - K_{X/Y} = \ru{D_{-}} - K_{X} + f^{*}K_{Y}.
$$
As $X$ is compact, $E$ is a $\BQ$-line bundle.
We see $(\Del_D+M)-E \sim_\BQ f^*L$.
We take a section 
$$
	\sg=s_{\ru{D_{-}}} 
	\in H^{0}(X,\CO_{X}(\ru{D_{-}})) = H^{0}(X, \CO_{X}(K_{X/Y}+E))
$$ 
whose divisor is $\ru{D_{-}}= \ru{-D}$ (to be more precise the divisor of $s_{\ru{D_{-}}}|_{X_{reg}}$ is $\ru{-D}|_{X_{reg}}$).
We take a (root of) section $s_{\Del_{D}}$ whose divisor is $\Del_{D}$, 
and any Hermitian metric $h_{\Del_{D}}$ on the $\BQ$-line bundle $\Del_{D}$.
Let $h_M$ be a Hermitian metric on the $\BQ$-line bundle $M$, and $\vph$ be a quasi-psh function on $X$.
(A model case is that $\vph$ has a vanishing Lelong number everywhere.)
Let 
$$	
	g = \frac{ h_{\Del_{D}} }{ |s_{\Del_{D}}|_{h_{\Del_{D}}}^{2} } e^{-\vph} h_{M}
$$
be a singular Hermitian metric of $\Del_{D}+M$.

(2)
We take a small local coordinate $V$ of $Y$.
We may suppose $V$ is biholomorphic to a ball in $(\BC^{m}, t=(t_1,\ldots,t_m))$, and $V$ is a relatively compact subdomain in a larger domain on which all of our relevant objects are trivialized.
Denote by $X_{V}= f^{-1}(V)$.
We can suppose $L\sim_{\BQ}0$ on $V$.
We have 
$$
	E_{V}:=E|_{X_{V}}\sim_{\BQ} (\Del_{D}+M)|_{X_{V}}
$$ 
on $X_{V}$.
A singular Hermitian metric on $E_{V}$ induces a singular Hermitian metric on a $\BQ$-line bundle $(\Del_{D}+M)|_{X_{V}}$ and vice versa.
We give $E_{V} (\sim_{\BQ} (\Del_{D}+M)|_{X_{V}})$ a singular Hermitian metric 
$$
	g_{V} = g|_{X_{V}}.
$$
We then have a possibly singular relative $(n,n)$-form 
$
	\w_{V} = g_{V} c_{n}\sg\wed\ol\sg
$
on $X_{V}$, and a function $F_{V}: V \sm Q \to [0,+\infty]$ given by 
$$
	F_{V}(y)=\int_{X_{y}} \w_{V}. \ \
$$ 

(3)
We take a Hermitian metric $h_{E}$ on the $\BQ$-line bundle $E$, and let 
$$
	g_{E} =  h_{E} \ \frac{ e^{-\vph} }{ |s_{\Del_{D}}|_{h_{\Del_{D}}}^{2} }
$$ 
be a singular Hermitian metric on $E$.
Then we have a possibly singular relative $(n,n)$-form 
$\w = g_{E} c_{n}\sg\wed\ol\sg$ on $X$ and a function $F_{\w} : Y \sm Q \to  [0,+\infty]$ given by 
$$
	F_{\w}(y)=\int_{X_{y}} \w.
$$
\qed\end{no'}

If we change a Hermitian metric $h_E$ on $E$ to another $\wtil h_E$, we have another $\wtil F_{\w} : Y \to [0,+\infty]$.
As we have $c^{-1}h_E \le \wtil h_E \le ch_E$ as a Hermitian metric on $E$ for some constant $c\in \BR_{\ge 1}$, we have $c^{-1}F_{\w} \le \wtil F_{\w} \le c F_{\w}$.
In particular, the ``singularities'' of $F_{\w}$ depend only on the zero divisors of $\sg$ and of $s_{\Del_{D}}$, and on the singularities of $\vph$ and of the map $f$.
The difference between (2) and (3) above is the choice of metrics $(E_{V}, g_{V})$:\ semi-global/local over $Y$, or $(E, h_{E})$:\ global.
The singularities of $F_V$ in (2) and $F_\w|_V$ in (3) are ``equivalent''.

The fiber integral $F_{\w}(y)$ and $F_V(y)$ are computed on every orbi-fold local chart of $X$, and are reduced to a computation of fiber integrals on a smooth total space ($\tU_{k}$ as in \ref{intForb}\,(3)).
Thus we have

\begin{thm'}\label{L1orb}
Let everything be as in \ref{tKsp} (with quasi-smooth $X$) and suppose $\Del_{Y}=0$.
Let $\vph$ be a quasi-psh function on $X$ with a vanishing Lelong number everywhere on $X$.
We follow the notations in \ref{intForb2}\,(3).
Let $y_{0}\in Y$ and let $V \subset (\BC^m,t=(t_1,\ldots,t_m))$ be a local coordinate centered at a point $y_0\in Y$.
Then

(1) Upper estimate.
$F_{\w} /|\prod_{j=1}^m t_j|^{c} \in L^1_{loc}(V,0)$ for any $c\in (0,2)$.
In particular, if $-\log F_{\w}$ is quasi-psh on a neighborhood of $y_{0}$,
one has $\nu(-\log F_{\w}, y_{0})=0$.

(2) Lower estimate.
There exist positive constants $\del(<1)$ and $c_{1}$ such that, for any $t\in V \sm Q$ with $|t|<\del$, $F_{\w}(t) \ge c_{1}$ holds.
\end{thm'}

We note that, for every orbi-fold local chart $\pi:\tU\to U$ of $X$, 
$\pi^{*}\vph$ is quasi-psh on $\tU$ and $\nu(\pi^{*}\vph, \wtil x)=0$ for every $\wtil x \in \tU$, see \ref{Le=0}.

\subsection{Comparison of fiber integrals}

We will compare the singularities of a fiber integral $F_{\w}$ on the quasi-smooth total space $X$ in \ref{intForb2}\,(2) and the one on a smooth model $X'$.
Let $\ga:X' \to X$ be a resolution of singularities as in \ref{modif}.
We will study a certain fiber integral $F_{\w'}$ for $f'=f\circ\ga$, and will see the singularities of $F_{\w}$ and $F_{\w'}$ are ``equivalent''.
Curvature properties of $M$ are irrelevant.

\begin{no'}\label{intFsm}
Let everything be as in \ref{tKsp}.
We take a log-resolution $\ga:X'\to X$, and let $K_{X'}+D'=\ga^*(K_X+D)$ as in \ref{modif}.
We then take objects $E, \sg, \ldots$ in \ref{intForb} for $f'=f\circ\ga:X'\to Y$ as follows.
 
(1)
We take
$$
	E' = \ru{D'_{-}} - K_{X'/Y},
$$
which is a line bundle on $X'$.
We have $(\Del_{D'}+M')-E' \sim_\BQ {f'}^*L$.
We note that $E' \ne \ga^* E$ in general, where $E$ is the orbi-line bundle in \ref{intForb2}.
We take a section 
$$
	\sg'=s_{\ru{D'_{-}}} 
	\in H^{0}(X', \ru{D'_{-}}) = H^{0}(X', K_{X'/Y}+E')
$$ 
whose divisor is $\ru{D'_{-}}= \ru{-D'}$.
We take a (root of) section $s_{\Del_{D'}}$ whose divisor is $\Del_{D'}$, and take any Hermitian metric $h_{\Del_{D'}}$ on the $\BQ$-line bundle $\Del_{D'}$.
We take the pull-back metric $h_{M'} = \ga^*h_M$ on the $\BQ$-line bundle  $M'=\ga^*M$, and let $\vph' = \ga^*\vph$ be a quasi-psh function on $X'$, where $h_{M}$ and $\vph$ are the one in \ref{intForb2}\,(1).
Let 
$$	
	g' = \frac{ h_{\Del_{D'}} }{ |s_{\Del_{D'}}|_{h_{\Del_{D'}}}^{2} } e^{-\vph'} h_{M'}
$$
be a singular Hermitian metric of $\Del_{D'}+M'$.
Note that $g' \ne \ga^*g$ in general.

(2)
We take a small local coordinate $V$ of $Y$ as in \ref{intForb}\,(2), and let  $X'_{V}= {f'}^{-1}(V)$.
We have 
$$
	E'_{V}:=E'|_{X'_{V}}\sim_{\BQ} (\Del_{D'}+M')|_{X'_{V}}
$$ 
on $X'_{V}$.
We give $E'_{V} (\sim_{\BQ} (\Del_{D'}+M')|_{X'_{V}})$ a singular Hermitian metric 
$$
	g'_{V} = g'|_{X'_{V}}.
$$
We have a possibly singular relative $(n,n)$-form 
$
	\w'_{V} = g'_{V} c_{n}\sg'\wed\ol\sg'
$
on $X'_{V}$.
Denote $Y_{f'\,\text{sm}}=\{y\in Y;\ \text{$f'$ is smooth over $y$}\}$.
We consider a function 
$$
	F'_{V}: V \sm Y_{f'\,\text{sm}} \to [0,+\infty], \ \ 
	F'_{V}(y)=\int_{X'_{y}} \w'_{V}. 
$$ 
\qed\end{no'}

As we do not know if $f' : X'\to Y$ is well-prepared, we can not estimate $F'_{V}$ directly as in \ref{L1}.

\begin{lem'}\label{equiv}
Take a small local coordinate $(V,t=(t_{1},\ldots,t_{m}))$ of $Y$ as in \ref{intForb2}\,(2).
We obtain fiber integrals $F_{V}$ as in \ref{intForb2}\,(1) and (2), and $F'_{V}$ as in \ref{intFsm}.

Then, there exists a constant $c\in \BR_{\ge 1}$ such that $c^{-1}F_{V} \le F'_{V} \le cF_{V}$ as a function on $(V \sm Q) \cap Y_{f'\,\text{\rm sm}}$.
In particular, $F'_V$ also satisfies \ref{L1} in an appropriate sense.
\end{lem'}

\begin{proof}
We will see that our claim is exactly a consequence of the formula $K_{X'}+D' = \ga^{*} (K_{X}+D)$.
It is important to note that the integrand $\w'_V$ of $F'_{V}$ (to be more precise $\w'_V \wed {f'}^*|dt|^2$) is defined on $X'_{V}={f'}^{-1}(V)$, not only on $X'_{V} \sm Y_{f'\,{\rm sm}}$.

We cover $X'_{V}= {f'}^{-1}(V)$ by a finite number of local charts $U'_{\ell}$ such that each $\ga(U'_{\ell})$ contained in an orbi-fold local chart of $X$ (as in \ref{intForb}). 
Let $U' \subset X$ be one of $\{U'_{\ell}\}$, and let $\pi : \tU \to U=\tU/G$ be an orbi-fold local chart of $X$ with $\ga(U') \subset U$.
We denote $V_0 = (V \sm Q) \cap Y_{f'\,\text{\rm sm}}$.

We will denote $\w, \w', F_\w$ and $F_{\w'}$ instead of $\w_V, \w'_V, F_V$ and $F'_{V}$ respectively.
We compare $(m,m)$-forms $F_{\w}|dt|^{2}$ and $F_{\w'}|dt|^{2}$ on $V_{0}$.
We have a possibly singular $(n+m,n+m)$-form $\w' \wed {f'}^{*} |dt|^{2}$ on $X'_{V}={f'}^{-1}(V)$ (not only on ${f'}^{-1}(V_{0}$)).
We have $F_{\w'}|dt|^{2} = f'_{*}(\w' \wed {f'}^{*} |dt|^{2})$ by the fiberwise integrations on which $f'$ is smooth.
Following the notations in Step 3 in the proof of \ref{L1} (where the total space is smooth), we have $\sg'({f'}^{*}dt)|_{U'} = \sg'_{U'} dx'$ for some $\sg'_{U'}\in H^{0}(U',\CO_{X'}) (\cong H^{0}(U',E'))$, where $E' = \ru{D'_{-}} - K_{X'/Y}$ in \ref{intFsm} and where $dx'$ is a local frame of $K_{X'}|_{U'}$.
Then
$$
	\w' \wed {f'}^{*} |dt|^{2}
	= A_{1}' \frac{ e^{-\vph'} }{ |s_{\Del_{D'}}|^{2} } \big| \sg'({f'}^{*} dt) \big|^{2}
	= A_{2}'  e^{-\vph'}  \frac{ |\sg'_{U'}|^{2} }{ |s_{\Del_{D'}}|^{2} } |dx'|^{2}
$$
on $U'$, where $A_{1}', A_{2}' \in C^{\infty}(\ol {U'}, \BR_{>0})$ which reflect local trivializations.
We have a similar expression for $F_{\w}|dt|^{2}$ formally, namely 
$$
	\w \wed {f}^{*} |dt|^{2}
	= A_{1} e^{-\vph} \frac{ |\sg_{U}|^{2} }{ |s_{\Del_{D}}|^{2} } |dx|^{2}
$$
on $U$, where $A_{1} \in C^{\infty}(\ol {U}, \BR_{>0})$. 
However as $K_{X}$ is merely $\BQ$-Cartier, we do not have a precise meaning of $dx$ (a local frame of $K_{X}$), nor $\sg_{U} \in H^{0}(U,\CO_{X})$.
Things are justified after passing to a local orbi-fold covering or passing to a high power to make every divisor to be Cartier.
As we need to compare with objects on $X'$, we adopt the latter way.
We take a large and divisible integer $p$, so that $pK_{X}$ is Cartier, for example, and take a local frame $k_{U,p} \in H^{0}(U,pK_{X})$.
Then $|k_{U,p}|^{2/p}$ can be seen as a positive measure on $U$.
To be more precise, we take a reference Hermitian metric $h_{p}$ on a line bundle $\CO_{X}(pK_{X})$, and put $|k_{U,p}|^{2}=|k_{U,p}|_{h_{p}}^{2} \cdot h_{p}^{-1}$ (as usual).
Alternatively, thinking $\pi^{*}k_{U,p} \in H^{0}(\tU, pK_{\tU})$, $|\pi^{*}k_{U,p}|^{2/p}$ is a positive measure on $\tU$ and is $G$-invariant, and hence it induces $|k_{U,p}|^{2/p}$ on $U$.
Let us look at $\sg^{p}\in H^{0}(X, pK_{X/Y}+pE) = {\rm Hom}_{\CO_{X}}(f^{*}pK_{Y}, pK_{X}+pE)$.
We have $\sg^{p}(f^{*}dt^{\ot p})|_{U} = \sg_{U,p}k_{U,p}$ for some $\sg_{U,p}\in H^{0}(U,\CO_{X}) (\cong H^{0}(U,pE))$.
Then 
$$
	(\w \wed {f}^{*} |dt|^{2})^{p}
	= A_{2} e^{-p\vph} \frac{ |\sg_{U,p}|^{2} }{ |s_{\Del_{D}}|^{2p} } 
	|k_{U,p}|^{2},
$$
where $A_{2} \in C^{\infty}(\ol {U}, \BR_{>0})$. 

We can write $\ga^{*}k_{U,p} = J_{\ga,p} (dx')^{\ot p}$ for some meromorphic function $J_{\ga,p}$ on $U'$ whose divisor is $\divisor J_{\ga,p}=pK_{X'/X}|_{U'}$. 
(We need to take a power $p$ to formulate this relation.)
In view of the formula of change of variables, we write as
$|\ga^{*} k_{U,p}|^{2/p}= |J_{\ga,p}|^{2/p} |dx'|^{2} $.
Let us look at the ratio on $U'$:
$$
	\frac{ (\w' \wed {f'}^{*} |dt|^{2})^{p} }{ \ga^{*}(\w \wed {f}^{*} |dt|^{2})^{p} }
	= ({A_{2}'}^{p} / \ga^{*} A_{2})
	 \frac{ |\sg'_{U'}|^{2p} }{ |s_{\Del_{D'}}|^{2p} } \cdot
	 \frac{ \ga^{*} |s_{\Del_{D}}|^{2p} } { \ga^{*}|\sg_{U,p}|^{2} }
	\cdot \frac{ |dx'|^{2p} }{ |\ga^{*} k_{U,p}|^{2} }.
$$
Note that the right hand side has no ``zeros'' or ``poles''.
In fact, the ``divisor'' of the right hand side is $-pD' + \ga^{*}pD - pK_{X'/X} = 0$, because of $K_{X'}+D'=\ga^{*}(K_{X}+D)$.
Here, we used that $\divisor \sg'_{U'} - \divisor s_{\Del_{D'}}
= \ru{D'_{-}} - (D'_{+} + \ru{D'_{-}} - D'_{-}) = - D'$ on $U'$, 
and similarly $\divisor \sg_{U,p} - \divisor s_{\Del_{D}}^{p}
= - pD$ on $U$.
Hence we have 
$$
	\frac{ \w' \wed {f'}^{*} |dt|^{2})^{p} }{ \ga^{*}(\w \wed {f}^{*} |dt|^{2})^{p} }
	\in C^{\infty}(\ol{U'},\BR_{>0}),
$$ 
and 
$\w' \wed {f'}^{*} |dt|^{2} = A \ga^{*}(\w \wed {f}^{*} |dt|^{2})$ for some $A \in C^{\infty}(\ol{U'},\BR_{>0})$ pointwise on each $U'=U'_{\ell}$.
Thus we have $c\in \BR_{\ge 1}$ such that
$$
	c^{-1} \ga^{*}(\w \wed {f}^{*} |dt|^{2})
	\le \w' \wed {f'}^{*} |dt|^{2} 
	\le c \ga^{*}(\w \wed {f}^{*} |dt|^{2}) 
$$
pointwise on $X'_{V}={f'}^{-1}(V)$.
Then, by taking the push-forward $f'_{*} = f_{*} \circ \ga_{*}$ over $V_0$, namely taking fiberwise integrals, we have
$$
	c^{-1} F_{\w}|dt|^{2}
	\le F_{\w'}|dt|^{2}
	\le c F_{\w}|dt|^{2}
$$
on $V_0 = (V \sm Q) \cap Y_{f'\,\text{\rm sm}}$.
\end{proof}


\section{Semi-positivity of the moduli part}

We shall prove \ref{tKsp}.
We construct a ``canonical'' singular Hermitian metric on the $\BQ$-line bundle $M_{Y}=L$.
We adapt a relative Bergman kernel construction of adjoint type bundles, developed in \cite{BP}, \cite{BP10}, \cite{PT}, \cite{CH} for it.
Theorem \ref{tKsp} can be stated more concretely as in \ref{can sHm} on a smooth model of $X$.
We follow \cite[\S 3]{CH} for a general outline of the proof. 

\begin{setup}\label{mKsp}
We follow the construction in \ref{modif}.

(1)
Starting from the situation in \ref{tKsp}, we take a log-resolution $\ga:X'\to X$ of the pair $(X,P)$, and put $f'=f\circ \ga : X' \to Y$.
We write as $	K_{X'}+D' = \ga^{*}(K_{X}+D)$ for a $\BQ$-divisor $D'$ on $X'$.
The pair $(X',D')$ satisfies (1)--(3) in \ref{modif}.
In particular, 
$$
	K_{X'}+D'+M' \sim_{\BQ} {f'}^{*}(K_{Y}+L),
$$ 
where $M'=\ga^{*}M$.
We set $D' = \sum_{j} d'_{j} P'_{j}$, and consider the discriminant part $\Del_{Y}'$ and the moduli part $M_{Y}' = L - \Del_{Y}'$ for $f'$ and $Q$.
By \ref{modif}, $\Del_{Y}' = \Del_{Y}$ and $M_{Y}' = M_{Y}$.
As we can suppose that $\Del_{Y}=0$, we have $\Del_{Y}'=0$ and $M_{Y}' = L$.

(2)
Let $\Sigma'\subset Y$ be the minimum Zariski closed subset such that $Y\sm \Sigma' = \{y\in Y;\ f'$ is smooth over $y$ and $h^0(X'_y, \ru{D'_-})=1 \}$.

(3)
By our assumption, there exist a Hermitian metric $h_{M}$ on $M$ and a quasi-psh function $\vph$ on $X$ such that a singular Hermitian metric $e^{-\vph} h_{M}$ of $M$ has semi-positive curvature and that $\vph$ has a vanishing Lelong number everywhere on $X$.
Let $h_{M'}=\ga^*h_M$ and $\vph'=\ga^*\vph$.
We write $D'=\Del_{D'} - \ru{D'_{-}}$ as in \ref{nota1}.
Let $h_{\Del_{D'}}$ be a Hermitian metric on $\Del_{D'}$, and $s_{\Del_{D'}}$ be a root of section whose divisor is $\Del_{D'}$.
We let 
$$
	g' = \frac{ h_{\Del_{D'}} }{ |s_{\Del_{D'}}|_{h_{\Del_{D'}}}^{2} }
	 e^{-\vph'}h_{M'}
$$ 
be the induced singular Hermitian metric on $\Del_{D'}+M'$ with semi-positive curvature.
\qed\end{setup}

Noting the relation $K_{X'/Y}+D'+M'=K_{X'/Y}+\Del_{D'}+M' - \ru{D'_{-}}$ in mind, we first consider an adjoint type $\BQ$-line bundle $K_{X'/Y}+\Del_{D'}+M'$.
We note that $\CI(g')=\CO_{X'}$, by \ref{open}, as the pair $(X',\Del_{D'})$ is klt and $\nu(\vph',x')=0$ for any $x'\in X'$.

\begin{thm}\label{can sHm}
Let everything be as in \ref{mKsp}.

(1) The $\BQ$-line bundle $K_{X'/Y}+\Del_{D'}+M'$ admits a singular Hermitian metric $b^{-1}$ whose curvature current satisfies $\frac{\ai}{2\pi}\Theta_{b^{-1}} \ge [\,\ru{D'_{-}}\,]$, where  $[\,\ru{D'_{-}}\,]$ is the current of integration over the effective cycle $\ru{D'_{-}}$.
More precisely, the restriction $b|_{X'_{y}}$ is the Bergman kernel metric on $-(K_{X'/Y}+\Del_{D'}+M')|_{X'_{y}}$ with respect to $(\Del_{D'}+M', g')|_{X'_{y}}$ for every $y \in Y \sm \Sigma'$.

(2) The $\BQ$-line bundle ${f'}^{*}L$ admits a singular Hermitian metric $\wtil h$ with semi-positive curvature, which is canonically constructed via the isomorphism ${f'}^{*}L \sim_{\BQ} K_{X'/Y}+\Del_{D'}+M' - \ru{D'_{-}}$ and $b^{-1}$.
The corresponding singular Hermitian metric $h_{Y}$ on $L$ (such that ${f'}^{*}h_{Y}= \wtil h$) has semi-positive curvature.

(3) 
Let us denote, on every sufficiently fine local chart $V$ on $Y$, $h_{Y}|_{V} = e^{-\psi_{V}}$ with $\psi_{V} \in L^{1}_{loc}(V,\BR)$ and $\psi_{V}$ is psh. 
Then $h_{Y}$ is given by a fiber integral for $f'$ as in \ref{intFsm}:\ $h_{Y} = F'_{V}$ on $V \sm \Sigma'$.
The function $-\log F'_{V}$ on $V \sm \Sigma'$ extends to a psh function $\psi_{V}$ on $V$ as a consequence.

 (4) 
(We now use the assumption that $\vph$ has a vanishing Lelong number everywhere on $X$.)
In the setting of (3) above,  one has $\nu(\psi_{V},y)=0$ for every $y \in V$.
In particular, $L$ is nef by \ref{L0nef}.
\end{thm}

\begin{proof}
{\it Step 1. Preliminary discussion}.
We let, as in \ref{intFsm},
$$
	E' = \ru{D'_{-}} - K_{X'/Y},
$$
which is a line bundle on $X'$.
We have $K_{X'/Y}+E'=\ru{D'_{-}}$ and $E' \sim_{\BQ} \Del_{D'}+M' - {f'}^{*}L$.
In particular, we can regard $E'$ to be pseudo-effective modulo ${f'}^{*}L$ and $\ru{D'_{-}}$ to be its adjoint line bundle.
We take an auxiliary Hermitian metric $h_{K'}$ of $K_{X'/Y}$, and put
$$
	h=h_{K'} h_{\Del_{D'}} h_{M'}
$$ 
a Hermitian metric on $K_{X'/Y}+ \Del_{D'}+M'$ with the curvature form
$$
	\theta = \frac{\ai}{2\pi}\Theta_{h}
	\in c_{1}(K_{X'/Y}+ \Del_{D'}+M').
$$
The curvature current of $g'$ on $\Del_{D'}+M'$ is
$$
	\frac{\ai}{2\pi} \Theta_{ g' } 
	= [\Del_{D'}] + \frac{\ai}{2\pi} \Theta_{e^{-\vph'}h_{M'}} \ge 0,
$$ 
where $[\Del_{D'}]$ stands for the current of integration over the effective $\BQ$-cycle $\Del_{D'}$.

We will consider restrictions of those $\BQ$-line bundles and metrics $g'$ and $e^{-\vph'}$ to fibers.
There are some fibers on which such restrictions do not behave well.
There are proper Zariski closed subsets $\Sigma_{1}, \Sigma_{2} \subset Y$ and pluri-polar subsets $\Sigma_{3}, \Sigma_{4} \subset Y$ with the following properties (o)--(iv).

\smallskip
(o) 
$\Sigma_{1}\subset\Sigma_{2}\subset\Sigma_{3} \subset \Sigma_{4} \subset Y$.
We actually have $\Sigma_2=\Sigma'$ in the statement.

\smallskip
(i) Let $\Sigma_{1}=Y \sm Y_{f'\,\text{sm}}$.
In particular, $f'$ is smooth over $Y \sm \Sigma_{1}$.

\smallskip
(ii) Let $\Sigma_{2}' = \{y \in Y ;\ h^{0}(X'_{y}, \ru{D'_{-}}) \ne 1 = {\rm rank}\,f_{*}\CO_{X'}(\ru{D'_{-}})\}$, 
and $\Sigma_{2}=\Sigma_{2}' \cup \Sigma_{1}$.
As $\CI(g') =\CO_{X'}$, we see that, if $y \in Y \sm \Sigma_{2}$, then $h^{0}(X'_{y}, \ru{D'_{-}}) = 1$ and $\CI(g')$ is flat over $y$, which is suffice to apply the base change theorem over $Y \sm \Sigma_{2}$.


\smallskip
(iii) Let $\Sigma_{3}' = \{y \in Y \sm \Sigma_{2};\ 
H^{0}(X'_{y}, \ru{D'_{-}} \ot \CI(g'|_{X'_{y}})) \subsetneqq H^{0}(X'_{y}, \ru{D'_{-}} \ot \CI(g')|_{X'_{y}})  \}$,
and $\Sigma_{3}=\Sigma_{3}' \cup \Sigma_{2}$.
We know $\CI(g') = \CO_{X'}$, but we use the notation $\CI(g')$ to remind the restriction property of multiplier ideal sheaves.
In particular, if $y \in Y\sm \Sigma_{2}$ and $g'|_{X'_{y}} \equiv + \infty$, then $y \in \Sigma_{3}'$.
We have also $h^{0}(X'_{y}, \ru{D'_{-}} \ot \CI(g'|_{X'_{y}}))=0$ and $h^{0}(X'_{y}, \ru{D'_{-}} \ot \CI(g')|_{X'_{y}})=1$ for $y\in \Sigma_{3}'$.
The set $\Sigma_{3}$ is pluri-polar in $Y$ (see (iv) below).

\smallskip
(iv) We set $\Sigma_{4}':=\{y \in Y\sm \Sigma_{2};\ \CI(g'|_{X'_{y}}) \subsetneqq \CI(g')|_{X'_{y}}=\CO_{X'_{y}}  \}$, and
$\Sigma_{4}=\Sigma_{4}' \cup \Sigma_{2}$.
This set $\Sigma_{4}$ is pluri-polar in $Y$ by \cite[2.9]{Xi}.
In particular, $\Sigma_{4}$ is a set of Lebesgue measure zero in $Y$ (this property has been observed in \cite{MZ}).
We see $\Sigma_{3}' \subset \Sigma_{4}'$. 
Hence $\Sigma_{3}$ is also pluri-polar in $Y$ (which is actually the content of \cite[Step 2.6.4]{Xi}).
The role of $\Sigma_{4}$ is auxiliary.

\smallskip

{\it Step 2. Localization over the base}.
Let $\{Y_{i}\}_{i\in I}$ be a locally finite open cover of $Y$ such that every $Y_{i}$ is small enough and biholomorphic to a ball in $\BC^{m}, m= \dim Y$.
Denote by $X'_{i}= {f'}^{-1}(Y_{i})$.
We can suppose $L\sim_{\BQ}0$ on every $Y_{i}$.
Hence, on every $X'_{i}$, we have 
$$
	E'_{i}:=E'|_{X'_{i}}\sim_{\BQ} (\Del_{D'}+M')|_{X'_{i}},
$$
and $K_{X'/Y}+E'_{i} \sim_{\BQ} K_{X'/Y}+\Del_{D'}+M'$ on $X'_{i}$.
A singular Hermitian metric on $E'_{i}$ (resp.\ $K_{X'/Y}+E'_{i}$) induces a singular Hermitian metric on a $\BQ$-line bundle $(\Del_{D'}+M')|_{X'_{i}}$ (resp.\ $K_{X'/Y} +\Del_{D'}+M'$ on $X'_{i}$) and vice versa.
Via $E'_{i} \sim_{\BQ} (\Del_{D'}+M')|_{X'_{i}}$, we can regard 
$$
	h_{i}:=h_{K'} h_{\Del_{D'}} h_{M'}|_{X'_{i}} = h|_{X'_{i}}
$$ 
as a Hermitian metric on $K_{X'/Y}+E'_{i}$ with $\frac{\ai}{2\pi}\Theta_{h_{i}}=\theta$ on $X'_{i}$.
We also give $E'_{i} (\sim_{\BQ} (\Del_{D'}+M')|_{X'_{i}})$ a singular Hermitian metric 
$$
	g'_{i} = g'|_{X'_{i}}.
$$
We have $\frac{\ai}{2\pi} \Theta_{g'_{i}} = [\Del_{D'}|_{X'_{i}}] + \frac{\ai}{2\pi}  \Theta_{e^{-\vph'}h_{M'}}  \ge 0$ and $\CI(g'_{i}) = \CO_{X'_{i}}$.

We can see that $\{h_{i}/g'_{i}\}_{i\in I}$ gives a global singular Hermitian metric on $K_{X'/Y}$, i.e., $h_{i}/g'_{i}=h_{j}/g'_{j}$ on $X'_{i}\cap X'_{j}$.
In fact, on each $X'_{i}$, 
$$
	\frac{h_{i}}{g'_{i}} 
	= h_{K'} h_{\Del_{D'}} h_{M'}
	  \Big( \frac{ h_{\Del_{D'}} }{ |s_{\Del_{D'}}|_{h_{\Del_{D'}}}^{2} } e^{-\vph'}h_{M'}  \Big)^{-1}
	= h_{K'}|s_{\Del_{D'}}|_{h_{\Del_{D'}}}^{2} e^{\vph'}.
$$

\smallskip
{\it Step 3. Fiberwise normalization}.
We take a section 
$$
	\sg'= s_{ \ru{D'_{-}} } 
	\in H^{0}(X', \ru{D'_{-}}) = H^{0}(X', K_{X'/Y}+E')
$$ 
whose divisor is $\ru{D'_{-}}$. 
By the base change theorem, we have
 $H^{0}(X'_{y}, K_{X'/Y}+E'_{i}) = \BC \sg'_{y}$ with $\sg'_{y}:=s_{ \ru{D'_{-}} }|_{X'_{y}}$ for $y\in Y_{i} \sm \Sigma_{2}$. 
For every $y\in Y_i \sm \Sigma_{2}$, we consider its $L^{2}$-norm with respect to $g'_{i}$, namely the fiber integral in our context:
$$
	\| \sg'_{y} \|_{g'_{i}}^{2} = F'_{i}(y)
	= \int_{X'_{y}} g'_{i} c_{n} \sg'_{y}\wed \ol{\sg'}_{y} \in (0,+\infty].
$$
This $\| \sg'_{y} \|_{g'_{i}} < +\infty$ if and only if 
$\sg'_{y} \in H^{0}(X'_{y}, \ru{D'_{-}} \ot \CI(g'|_{X'_{y}}))$. 
We can also see that, for $y \in Y_{i}\sm \Sigma_{2}$, $\| \sg'_{y} \|_{g'_{i}} = +\infty$ if and only if $y \in \Sigma_{3}'$.

For $y\in Y_{i}\sm \Sigma_{3}$ (namely $\| \sg'_{y} \|_{g'_{i}} < +\infty$), there exists $\tau_{i,y} \in H^{0}(X'_{y}, \ru{D'_{-}})$ such that
$$
	\int_{X'_{y}} g'_{i} c_{n} \tau_{i,y}\wed \ol {\tau_{i,y}} =1.
$$
Such $\tau_{i,y}$ is unique up to a multiplication of a complex number of unit norm.
In fact, we can take 
$$
	\tau_{i,y} := \frac{ \sg'_{y} \ }{  \| \sg'_{y} \|_{g'_{i}}  }
	\in H^{0}(X'_{y}, K_{X'/Y}+E'_{i})
$$
for $y \in Y_{i}\sm \Sigma_{3}$.
We take $\tau_{i,y}=0 \in H^{0}(X'_{y}, K_{X'/Y}+E'_{i})$ when $y \in Y_{i}\cap (\Sigma_{3}\sm \Sigma_{2})= Y_{i} \cap \Sigma_{3}'$ (namely $\| \sg'_{y} \|_{g'_{i}} = +\infty$).
Hence, the formulation $\tau_{i,y} := \frac{ \sg'_{y} \ }{  \| \sg'_{y} \|_{g'_{i}} }$ is still valid formally for any $y \in Y_{i} \sm \Sigma_{2}$.

\smallskip

{\it Step 3'. Role of fiber integral.}
We have a relative $(n,n)$-from $\w'_{i}=g'_{i}c_{n}\sg' \wed \ol{\sg'}$ on $X'_{i}$ and the fiber integral
$$
	F'_{i}(y) = \int_{X'_{y}} \w'_{i} \in [0,+\infty]
$$
for $y \in Y_{i} \cap Y_{\text{$f'$\,sm}}$.
We note that the integrand $\w'_{i}$ is defined on $X'_{i}$, not only on $X'_{i}
\sm {f'}^{-1}(\Sigma_{2})$.
	
\smallskip

{\it Step 4. Fiberwise Bergman kernel metric}.
We now apply \cite{BP} (\cite{PT}) for $f:X'_{i}\to Y_{i}$ and $K_{X'/Y}+E'_{i}$ with respect to $(E'_{i},g'_{i})$. 

We let $\tau_{i}=\{ \tau_{i,y} \}_{y\in Y_{i} \sm \Sigma_{2}}$ be a family of fiberwise holomorphic sections parametrized by $Y_{i} \sm \Sigma_{2}$.
We define an $\BR_{\ge 0}$-valued function $\gb_{i}$ on $X'_{i}\sm {f'}^{-1}(\Sigma_{2})$ by 
$$
	\gb_{i}(x') := |\tau_{i}|_{h_{i}}^{2}(x') = |\tau_{i,y}|_{h_{i}}^{2}(x')
$$ 
for $x' \in X'_{y}$ with $y=f'(x') \in Y_{i} \sm \Sigma_{2}$. 
This $\gb_{i}$ takes a value 0 exactly along $(\Supp D'_{-} \cup {f'}^{-1}(\Sigma_{3})) \cap (X'_{i}\sm {f'}^{-1}(\Sigma_{2}))$.
We let 
$$
	b_{i}=\gb_{i} h_{i}^{-1}. 
$$
On each fiber $X'_{y}$ for $y \in Y_{i} \sm \Sigma_{2}$ (even if $y \in \Sigma_{3}$), $b_{i}|_{X'_{y}}$ is a singular Hermitian metric on $(K_{X'_{y}}+E'_{i}|_{X'_{y}})^{*}$,  which is called the Bergman kernel metric with respect to the Hermitian line bundle $(E'_{i},g'_{i})|_{X'_{y}}$. 
We have $b_{i}|_{X'_{y}} \equiv 0$ for $y \in Y_{i}\cap\Sigma_{3}'$ by construction.
We consider its dual  
$$
	b_{i}^{-1} = \frac{h_{i}}{ |\tau_{i}|_{h_{i}}^{2}}
	= h_{i} \, e^{- \psi_{i}} \
	\text{ with } \ \psi_{i} 	= \log |\tau_{i}|_{h_{i}}^{2}
$$
defined on $X'_{i}\sm {f'}^{-1}(\Sigma_{2})$.
We note that $b_{i}$ does not depend on the choice of a reference metric $h_{i}$ but depends (only) on $g'_{i}$.

Then by \cite[0.1]{BP} (see also \cite[4.1.1]{PT}), on every local chart  of $X'_{i}$, a function $\log b_{i}(\sim \psi_{i})$ is psh on $X'_{i}\sm {f'}^{-1}(\Sigma_{2})$ and uniformly bounded from above around ${f'}^{-1}(\Sigma_{2})$.
(The last uniform upper-boundedness of $\psi_{i}$ also follows from \ref{equiv} and \ref{L1orb}\,(2) in our setting. 
See Step 5 below.)
Hence, by a Riemann type extension for psh functions $b_{i}^{-1}$ extends as a singular Hermitian metric of $K_{X'/Y}+E'_{i}$ on $X'_{i}$ (we still denote it by $b_{i}^{-1}= h_{i} \, e^{- \psi_{i}}$) with semi-positive curvature
$$
	\theta + dd^{c} \psi_{i} \ge 0
$$
on $X'_{i}$, where $dd^{c}=\frac{\ai}{2\pi}\rd\rdb$.
We regard $\psi_{i}$ as a quasi-psh function on $X'_{i}$ from now on.
Via $K_{X'/Y}+\Del_{D'}+M' \sim_{\BQ} K_{X'/Y}+E'_{i}$ on $X'_{i}$, we regard $b_{i}^{-1}$  as a singular Hermitian metric of $K_{X'/Y}+\Del_{D'}+M'$ on $X'_{i}$.

\smallskip
{\it Step 5. Singularities of the dual of the Bergman kernel metric}.
We note that a simple relation
$$
	e^{-\psi_{i}}  | s_{\ru{D'_{-}}} |_{h_{i}}^{2} (x')
	= \frac{ F'_{i}(f'(x')) }{ | s_{\ru{D'_{-}}}(x') |_{h_{i}}^{2} }
		| s_{\ru{D'_{-}}}(x') |_{h_{i}}^{2}
	= F'_{i}(f'(x'))
$$
holds for any $x' \in X'_{i}\sm {f'}^{-1}(\Sigma_{2})$. 
In particular, thanks to the lower bound in \ref{equiv} and \ref{L1orb}\,(2), a uniform upper bound 
$$
	\psi_{i} \le \log | s_{\ru{D'_{-}}} |_{h_{i}}^{2} + \log c_1^{-1} 
$$
holds on $X'_{i} \sm {f'}^{-1}(\Sigma_{2}\cup Q)$ for a constant $c_{1}>0$.
As both $\psi_{i}$ and $\log | s_{\ru{D'_{-}}} |_{h_{i}}^{2}$ are quasi-psh on $X'_{i}$, the relation $\psi_{i} \le \log | s_{\ru{D'_{-}}} |_{h_{i}}^{2} + \log c_1^{-1}$ still holds on $X'_{i}$. 

There would be several ways to see the last property: $\psi_{i} \le \log | s_{\ru{D'_{-}}} |_{h_{i}}^{2} + \log c_1^{-1}$ on $X'_{i}$ in general.
One way is to apply the following property.
Let $U$ be a complex manifold and let $\vph$ be a psh function on $U$.
Let $A \subsetneq U$ be a closed analytic subset.
Then, for each $z \in U$,
$$
	\vph(z) = \varlimsup_{w \to z} \{ \vph(w); \ w \in U \sm A \}
$$
holds (see \cite{NO}, the second line in the proof of (3.3.41) Theorem, for example).

Moreover, thanks to the estimate $\psi_{i} \le \log | s_{\ru{D'_{-}}} |_{h_{i}}^{2} + \log c_1^{-1}$, we see that the Lelong number of $\psi_{i}$ at each point is greater than or equal to that of the current $[\,\ru{D'_{-}}\,]$ (namely the multiplicity of the divisor).
Thus we have in fact
$$
	 \frac{\ai}{2\pi} \Theta_{h_{i} e^{-\psi_{i}}} = \theta+dd^{c}\psi_{i}  
	 \ge  [\,\ru{D'_{-}}\,]
$$
as currents on $X'_{i}$.

We now claim that $\psi_{i}=\psi_{j}$ on $X'_{i}\cap X'_{j}$.
It is shown in \cite[3.4, Step 4]{CH} (from page 12, line 20 to page 13, line 2) that,
for every $y \in Y_{i}\cap Y_{j} \sm \Sigma_{3}$, $\psi_{i}|_{X'_{y}} = \psi_{j}|_{X'_{y}}$ holds.

(Our $(E'_{i},g'_{i})$ corresponds to $(L_{i}, h_{L_{i}})$ in \cite[p.11, line 6 from the bottom]{CH}.
The metric $h_{i}$ is the same \cite[3.4, Step 2]{CH}.
The Zariski open set $U_{i,0}\subset U_{i} (\subset Y)$ in \cite[p.11, the bottom line]{CH} is our $Y_{i} \sm \Sigma_{2} \subset Y_{i}$.
The section $\tau_{i,y}$ corresponds to $s_{y}(= s_{y,i})$ in \cite[p.12, line 1]{CH}.
$\psi_{i}$ correspond to $2\vph_{i}$ in \cite[p.12, line 8]{CH}.)

Thus we have $\psi_{i}=\psi_{j}$ on $(X'_{i}\cap X'_{j}) \sm {f'}^{-1}(\Sigma_{3})$.
As both $\psi_{i}$ and $\psi_{j}$ are quasi-psh on $X'_{i}\cap X'_{j}$, and as ${f'}^{-1}(\Sigma_{3})$ is pluri-polar in $X'$, in particular the Lebesgue measure zero, $\psi_{i}=\psi_{j}$ on $X'_{i}\cap X'_{j}$
(\cite[\S K, 15 Theorem, 16 Corollary]{Gu}).
When $y \in Y_{i}\cap Y_{j} \cap (\Sigma_{3}\sm \Sigma_{2})$, we have in particular $\psi_{i}|_{X'_{y}} \equiv -\infty \equiv \psi_{j}|_{X'_{y}}$.

Thus $\{\psi_{i}\}_{i}$ defines a quasi-psh function on $X'$, which we denote by 
$$
	\psi \in L^{1}_{\rm loc}(X',\BR).
$$

\smallskip
{\it Step 6. Metric with semi-positive curvature on $M_{Y}=L$}.
We first have a global singular Hermitian metric
$ e^{-\psi} h_{K'}h_{\Del_{D'}} h_{M'}$ on $K_{X'/Y}+\Del_{D'}+M'$ with the curvature
$$
	\theta + dd^{c}\psi \ge [\,\ru{D'_{-}}\,]
$$
on  $X'$.
For a line bundle $\CO_{X'}(\ru{D'_{-}})$, we have a canonical singular Hermitian metric 
$1/ |s_{\ru{D'_{-}}}|^{2} = h_{\ru{D'_{-}}} / |s_{\ru{D'_{-}}}|_{ h_{\ru{D'_{-}}} }^{2}$
for any Hermitian metric $h_{\ru{D'_{-}}}$ of $\CO_{X'}(\ru{D'_{-}})$.
Then we consider a singular Hermitian metric $\wtil h$ on 
$K_{X'/Y}+\Del_{D'}+M'-\ru{D'_{-}} \sim_{\BQ}{f'}^{*}L$ given by
$$
	\wtil h = e^{-\psi} h_{K'}h_{\Del_{D'}} h_{M'} |s_{\ru{D'_{-}}}|^{2}.
$$
The curvature current is 
$$
	\frac{\ai}{2\pi} \Theta_{\wtil h} =\theta + dd^{c}\psi - [\,\ru{D'_{-}}\,] 
	\ge 0 
$$
on $X'$.
On each fiber $X'_{y}$, we have $({f'}^{*}L)|_{X'_{y}} \sim_{\BQ} \CO_{X'_{y}}$, and hence the semi-positively curved singular Hermitian metric on $({f'}^{*}L)|_{X'_{y}}$ needs to be constant (possibly $+\infty$).
Thus, we can find a singular Hermitian metric $h_Y$ on $L$ with semi-positive curvature such that ${f'}^{*}h_Y=\wtil h$.
We usually write it as $h_{Y}=f'_{*} \wtil h$.
At each point $x' \in X'_{i} \sm {f'}^{-1}(\Sigma_{2})$, we have
$$ 
	\wtil h (x')
	= e^{-\psi_{i}(x')} h_{i}(x')  |s_{\ru{D'_{-}}}(x')|^{2}
	=  \frac{ F'_{i}(f'(x')) }{ |s_{\ru{D'_{-}}} (x')|_{h_{i}}^{2} } h_{i}(x')  |s_{\ru{D'_{-}}}(x')|^{2}
	= F'_{i}(f'(x')).
$$
Needless to say, the right hand side $F'_{i}(f'(x'))$ is constant on each fiber (it can be $\equiv +\infty$ on some fibers).
Thus $h_{Y}(y) = F'_{i}(y) (= \| \sg'_{y} \|_{g'_{i}}^{2})$ for every $y \in Y_{i}\sm \Sigma_{2}$ and every $i$.

\smallskip
{\it Step 7. Bounding the singularities of the metric:\ the nefness of $M_{Y}=L$.}
Let us recall that $\vph$ has a vanishing Lelong number everywhere on $X$.
Then, \ref{equiv} and \ref{L1orb}\,(1) imply that the Lelong number of the local weight functions of $h_{Y}: \nu(-\log h_{Y},y)=0$ for every $y \in Y$.
As a conclusion, the singular Hermitian metric $h_{Y}=f'_{*}\wtil h$ on $M_{Y}=L$ is ``nef'' by \ref{L0nef}.
\end{proof}

\begin{rem}
There are two points we used \ref{L1} (or \ref{L1orb}) in the proof of \ref{can sHm}.

(1) \ref{L1}\,(2) is used in Step 5.
This follows from the existence of a positive lower bound of $e^{-\vph}$ and a condition $\Del_{Y'}=0$, which is a generic condition on the special fiber.

(2) \ref{L1}\,(1) is used in Step 7.
We used $\Del_{Y'} =0$ and $\nu(\vph',x')=0$ for every $x'\in X'$.
\qed
\end{rem}

\begin{rem}
Up to Step 6 in the proof above, our arguments (on the proof of \ref{can sHm}\,(1)--(3)) also work under a weaker assumption that the singular Hermitian metric $e^{-\vph'} h_{M'}$ has semi-positive curvature and the multiplier ideal $\CI(g')=\CO_{X'}$ over a non-empty Zariski open subset $Y_{M}$ of $Y$; $\CI(g')=\CO_{X'}$  on ${f'}^{-1}(Y_{M})$.
\qed\end{rem}


\section{On finite generation problem of log canonical rings for generalized pairs}\label{fg-section}

In this section, we apply twisted Kawamata's semi-positivity theorem (Theorem \ref{tKsp}) for the finite generation problem for some versions of canonical rings. 

\begin{dfn}\label{defi:gp}
A \textit{generalized pair} $(X,B+M)$ consists of
\begin{itemize}
\item a normal projective variety $X$, 
\item an effective $\mathbb{R}$-divisor $B$ on $X$, and 
\item a b-$\mathbb{R}$-Cartier b-divisor over $X$ represented by some projective birational morphism $\varphi: X' \to X$ and 
a nef $\mathbb{R}$-Cartier divisor $M'$ on $X'$
\end{itemize}
such that $M = \varphi_* M'$ and $K_{X}+B+M$ is $\mathbb{R}$-Cartier. 
\end{dfn}

\begin{dfn}
\label{def:gen-lc}
Let $(X, B+M)$ be a generalized pair.
Let $Y\rightarrow X$ be a projective birational morphism.
We can write
\[
K_Y+B_Y+M_Y=\pi^*(K_X+B+M),
\]
for some divisor $B_Y$ on $Y$.
The {\em generalized log discrepancy}
of $(X,B+M)$ at a prime divisor $E\subset Y$, denoted by $a_E(X,B+M)$
is defined to be $1-{\rm coeff}_E(B_Y)$.

A generalized pair $(X, B+M)$ is said to be {\em generalized log canonical} (or glc for short) if all its generalized log discrepancies are non-negative.
A generalized pair $(X,B+M)$ is said to be {\em generalized Kawamata log terminal} (or gklt for short) if all its generalized log discrepancies are positive.
In the previous definitions, if the analogous statement holds for $M=0$, then we drop the word ``generalized".

\end{dfn}

First, we see the following lemma by \cite{bchm}.

\begin{lem}\label{fg-canonical-bchm}Let  $(X, B+M)$ be a projective generalized klt pair such that $K_X+B+M$ is big. Assume that $l(K_X+B+M)$ is a Cartier divisor for an $l>0$. 

Then the ring 
$$ \bigoplus_{m \geq 0} H^0(X, ml(K_X+B+M))  
$$
is a finite generated $\mathbb{C}$-algebra. 
\end{lem}

\begin{proof}Note that by \cite[Proposition 5.1]{fm}, it is enough to show the lemma for a sufficiently large and divisible $l$. Taking a log resolution, we may assume that $(X, B)$ is simple normal crossing klt pair and $M$ is nef. Since $K_X+B+M$ is big, we take an effective divisor $E$ and an ample divisor $A$ such that $l(K_X+B+M) \sim_{\mathbb{Z}}E+A$. Since $M+\frac{1}{l} A$ is ample, we can choose an effective $\mathbb{Q}$-divisor $A'$ such that $(X, B+\frac{1}{l}(E+A'))$ is klt and $A' \sim_{\mathbb{Z}} lM+A$ for  a sufficiently large and divisible positive integer $l$. Since $$2l(K_X+B+M) \sim_{\mathbb{Z}} l(K_X+B+\frac{1}{l}(E+A')),$$
it is enough to show the finite generation of $$ \bigoplus_{m \geq 0} H^0(X, ml(K_X+B+\frac{1}{l}(E+A'))),
$$
which follows from \cite[Corollary 1.1.2]{bchm}. 
\end{proof}

Using  twisted Kawamata's semi-positivity theorem \ref{tKsp}, we prove the following theorem:

\begin{thm}\label{fg-canonical}Let  $(X, B+M)$ be a projective generalized klt pair such that a $\mathbb{Q}$-line bundle  $M'$ on a smooth higher birational model $X' \to X$ has a semi-positive singular Hermitian metric $h_{M'}$ with a vanishing  Lelong number for the potential at every point of $X'$. Assume that $l(K_X+B+M)$ is a Cartier divisor for an $l>0$. 

Then the ring 
$$ \bigoplus_{m \geq 0} H^0(X, ml(K_X+B+M))  
$$
is a finite generated $\mathbb{C}$-algebra. 
\end{thm}

\begin{proof}When $\kappa(K_X+B+M)=0$ or $\kappa(K_X+B_M) <0$,  $$\bigoplus_{m \geq 0} H^0(X, ml(K_X+B+M))  \simeq \mathbb{C}[t] \text{ or } \mathbb{C},
$$
respectively, where  $\mathbb{C}[t]$ is the polynomial ring with one variable.
 Thus we may assume that $\kappa(X, l(K_X+B+M)) > 0$. We consider the Iitaka fibration $f:X \dashrightarrow Y$of $l(K_X+B+M)$ and take a higher birational model $X' \to X$ induces a holomorphic morphism $f': X'\to Y$ which is birational equivalent to $f$. Thus, by replacing, we may assume that $(X,B)$ is a simple normal crossing klt pair admitting the Iitaka fibration $f:X \to Y$ of $K_X+B+M$, and that $M$ has a semi-positive singular Hermitian metric  with a  vanishing  Lelong number for the potential at every point of $X$. 
 
 Now, we apply \ref{wp}:~the toroidalization theorem due to Abramovich and Karu \cite{AK} for $f$.
We have a birational morphisms $\varphi:X' \to X$ and $\psi: Y' \to Y$, and a well prepared birational model $f': (X', B') \to Y'$ of $f$ such that $M'=\varphi^*M$ and
$$
	\varphi^*(K_X+B+M)=K_{X'}+B'+M'.
$$ 
We use the notation in twisted Kawamata's semi-positivity theorem \ref{tKsp}. Consider the divisor $\Delta_{Y'} ^{-}$. Since $B$ is effective, any components of $f'^*\Delta_{Y'} ^{-}$ are $\varphi$-exceptional divisors. Thus, $f'$ is also the Iitaka fibarion of $K_{X'}+f'^*\Delta_{Y'} ^{-}+ B'+M'$. Therefore the finite generation of $R(K_X+ B+M)$ is reduced to that of $R(K_{X'}+f'^*\Delta_{Y'} ^{-}+ B'+M').$ 
And that is equivalent to the finite generation of $R(K_{Y'}+ M_{Y'}+ \Delta^+_{Y'})$. 
Now  $(Y' ,\Delta_{Y'}^+) $ is klt and  $M_{Y'}$ is nef by Theorem \ref{tKsp}. Since $K_{Y'}+ M_{Y'}+ \Delta^+_{Y'}$ is big, the desired finite generation follows from Lemma \ref{fg-canonical-bchm}. 
\end{proof}

The following is an interesting corollary for even smooth projective varieties.

\begin{cor}Let  $X$ be a smooth projective variety such that anti-canonical bundle $\omega_X^{-1}$ has a  semipositive singular  metrics $h$ with a vanishing  Lelong number for the potential at every point of $X$.
Then, the anti-canonical ring 
$$ \bigoplus_{m \geq 0} H^0(X, \omega_X^{-m})  
$$
is a finite generated $\mathbb{C}$-algebra. 
\end{cor}

In the last, we discuss the finite generation of generalized canonical rings for generalized klt pairs in a general situation. Indeed we prove that for lower dimensional cases as follows:

\begin{thm}\label{fg-canonical II}Let  $(X, B+M)$ be a projective generalized klt pair such that $\mathrm{dim}\, X \leq 3$. 
 Assume that $l(K_X+B+M)$ is a Cartier divisor for an $l>0$. 

Then the ring 
$$ \bigoplus_{m \geq 0} H^0(X, ml(K_X+B+M))  
$$
is a finite generated $\mathbb{C}$-algebra. 
\end{thm}

\begin{proof}If $\kappa(X, l(K_X+B+M)) \leq 0$, the finite generation is clear as the proof of Theorem \ref{fg-canonical}.  We may assume that $1 \leq \kappa(K_X+B+M)\leq 2$ by Lemma \ref{fg-canonical-bchm}. 
 The case of $\kappa(K_X+B+M)=1$ is known since the finite generation is always true for a section ring of a Cartier divisor with Iitaka dimension one (cf. \cite[Lemma 4.4]{W}).
  When  $\mathrm{dim}\, X=2$, we complete the proof by Lemma \ref{fg-canonical-bchm} again. Thus we may assume that $\mathrm{dim}\,X =3$ and  $\kappa(K_X+B+M)= 2$. 
 As the proof of Theorem \ref{fg-canonical}, by taking the Iitaka fibration and log resolution, we may assume that there exists a projective surjective morphism $f:X \to Y$ with connected fibers such that $(X, B+M)$ is generalized klt and $\mathrm{dim}\, Y=2$. For a general fiber $F$ of $f$, it holds that $\kappa((K_X+B+M)|_F)=0$ since the property of Iitaka fibrations. 
Since $F$ is a smooth projective curve, we see that $(K_X+B+M)|_F \sim_{\mathbb{Q}}0.$ Thus  we can apply  \cite[Theorem 2.20]{filis}.

  Then there exists $B_Y$ and $M_Y$ such that $(Y, B_Y+M_Y)$ is generalized klt and the finite generation for $K_X+B+M $ is equivalent for that for a big divisor $K_Y+B_Y+M_Y$ as the proof of Theorem \ref{fg-canonical}.  Since $K_Y+B_Y+M_Y$ is big, the desired finite generation follows from Lemma \ref{fg-canonical-bchm}. 
\end{proof}



\begin{thebibliography}{Am05}		
\bibitem[AK]{AK}
D.~ Abramovich and K.~ Karu, 
{\it Weak semistable reduction in characteristic 0},
Invent.~ Math.~ {\bf 139} (2000) 241--273.

\bibitem[Am04]{Ajdg}
F.~ Ambro, {\it Shokurov's boundary property},
J.~ Differential Geom.\ {\bf 67} (2004) 229--255.

\bibitem[Am05]{Acomp}
F.~ Ambro, {\it The moduli b-divisor of an lc-trivial fibration},
Compositio Math.~ {\bf 141} (2005) 385--403.

\bibitem[Be]{Be}
B.~ Berndtsson,
{\it The openness conjecture for plurisubharmonic functions}, arXiv:1305.5781.

\bibitem[BP08]{BP}
B.~ Berndtsson and M.~ P\u{a}un,
{\it Bergman kernels and the pseudoeffectivity of relative canonical bundles},
Duke Math.\ J.\ {\bf 145} (2008) 341--378.

\bibitem[BP10]{BP10}
B.~ Berndtsson and M.~ P\u{a}un,
{\it Bergman kernels and subadjunction},
arXiv:1002.4145v1 [math.AG].

\bibitem[BCHM]{bchm}
C.~Birkar, ~P. Cascini, C. D. ~Hacon, and J. ~$\mathrm{M^{c}}$Kernan, 
{\textit{Existence of minimal models for varieties of log general type}}, 
J. Amer. Math. Soc. {\bf{23}} (2010), 405--468.

\bibitem[CM]{CM}
D.~ Coman and G.~ Marinescu, {\it Convergence of Fubini-study currents for orbifold line bundles}. Internat.~ J.~ Math.~ {\bf 24} (2013), 1350051, 27 pp.


\bibitem[D92]{Dem92}
J.-P.~Demailly,  
\textit{Regularization of closed positive currents and intersection theory. }
J.~ Algebraic Geom.~ {\bf 1} (1992) 361--409. 

\bibitem[D12]{Dem}
J.-P.~Demailly,   
{\textit {Analytic methods in algebraic geometry}}, 
Surv.~ Mod.~ Math., {\bf 1},
International Press, Somerville, MA; Higher Education Press, Beijing, 2012. viii+231 pp.

\bibitem[CH]{CH}
J.~Cao and A.~H\"oring, {\it Rational curves on compact K\"ahler manifolds},
J.\ Differential Geom.\ {\bf 114} (2020) 1--39.

\bibitem[Fa]{Fa}
C.~Favre, {\it Note on pull-back and Lelong number of currents},
Bull.~Soc.~Math.~France {\bf 127} (1999) 445--458.

\bibitem[Fil]{filipazzi}
S.~Filipazzi, {\it On a generalized canonical bundle formula and generalized adjunction}, Ann.~ Sc.~ Norm.~ Super.~ Pisa Cl.~ Sci.~ (5) {\bf 21} (2020) 1187--1221. 

\bibitem[FiS]{filis}
S.~Filipazzi and R.~Svaldi,   {\it  On the connectedness principle and dual complexes for generalized pairs},
Forum Math. Sigma 11 (2023), Paper No. e33, 39 pp.

\bibitem[FF]{FF}
O.~ Fujino and T.~ Fujisawa, {\it On semipositivity theorems},
Math.~ Res.~ Lett.~ {\bf 26} (2019) 1359--1382.

\bibitem[FG]{FG}
O.~ Fujino and Y.~Gongyo, 
{\it On the moduli b-divisor of an lc-trivial fibrations},
Ann.~Inst.~ Fourier {\bf 64} (2014) 1721--1735.

\bibitem[FM]{fm}
O.~ Fujino and S.~ Mori, {\it A canonical bundle formula},  J.~ Differential Geom.~ {\bf 56}  (2000) 167--188.

\bibitem[G]{gon}
Y.~Gongyo, {\it On weak Fano varieties with log canonical singularities}, J.~ Reine Angew.~ Math.~ {\bf 665} (2012) 237--252.

\bibitem[GZ]{GZ}
Q.~ Guan and X.~ Zhou,
{\it A proof of Demailly's strong openness conjecture}, Ann.~ of Math.~{\bf 182} (2015) 605--616.

\bibitem[Gu]{Gu}
R.~Gunning, {\it Introduction to holomorphic functions of several variables},
Vol.~1. Wadwhorth and Brooks/Cole, Math.~ Ser.~
Wadsworth and Brooks/Cole Advanced Books and Software, Pacific Grove, CA, 1990. xx+203 pp.

  \bibitem[HP]{hp}
C. D.~ Hacon and M.~ P\u{a}un, {\it On the Canonical Bundle Formula and Adjunction for Generalized Kaehler Pairs}, arXiv:2404.12007 [math.AG].

\bibitem[HPS]{HPS} 
C. D. ~Hacon, M.~Popa and Ch.~Schnell,
{\it Algebraic fiber spaces over abelian varieties: around a recent theorem by Cao and P\u{a}un},
Local and global methods in algebraic geometry, 143--195, Contemp.~ Math., {\bf 712}, Amer.~ Math.~ Soc., Providence, RI, 2018.

\bibitem[K98]{Ksub}
Y.~ Kawamata, {\it Subadjunction of log canonical divisors, II},
Amer.\ J.\ Math.\ (1998) {\bf 120}, 893--899.

\bibitem[K15]{Ka15}
Y.~ Kawamata, {\it Variation of mixed Hodge structures and the positivity for algebraic fiber spaces}, ASPM 65 (2015), Algebraic Geometry in East-Asia, --- Taipei 2011, pp.\, 27--57.

\bibitem[Ki19]{Kim2}D.~ Kim,  {\it Canonical bundle formula and degenerating families of volume forms}, arXiv:1910.06917v4.

\bibitem[KM]{komo}
J.~Koll\'ar and S.~Mori, 
{\textit{Birational geometry of algebraic varieties}}, 
Cambridge Tracts in Math.,134 (1998).

\bibitem[MZ]{MZ}
X.~ Meng and X.~ Zhou, {\it On the restriction formula},
J.~ Geom.~ Anal.~ {\bf 33} (2023), Paper No.~ 369, 30 pp.

\bibitem[NO]{NO}
J.~ Noguchi and T.~ Ochiai, {\it Geometric function theory in several complex variables}, 
Translated from the Japanese by Noguchi,
Transl.~ Math.~ Monogr., {\bf 80}
American Mathematical Society, Providence, RI, 1990. xii+283 pp.

\bibitem[PT]{PT}
M.~ P\u{a}un and S.~ Takayama,
{\it Positivity of twisted relative pluricanonical bundles and their direct images},
J.~ Algebraic Geom.~ {\bf 27} (2018) 211--272.

\bibitem[Pe]{Pe}
Th.~Peternell, {\it Pseudoconvexity, the Levi problem and vanishing theorems},
Encyclopaedia Math.~ Sci., {\bf 74}, Springer-Verlag, Berlin, 1994, 221--257.

\bibitem[Ta]{Ta}
S.~ Takayama,
{\it Singularities of $L^{2}$-metric in the canonical bundle formula},
Amer.\ J.\ Math.\ {\bf 144} (2022), 1725--1743.

\bibitem[W]{W}
J.~Waldron, 
 {\it Finite generation of the log canonical ring for 3-folds in char $p$}, Math. Res. Lett. 24 (2017), no. 3, 933--946.

\bibitem[Xi]{Xi}
M.~ Xia,
{\it Analytic Bertini theorem},
Math.~ Z.~ {\bf 302} (2022), 1171--1176.
\end{thebibliography}
\end{document}